\documentclass[12pt,a4paper,reqno]{amsart}
\usepackage{amsmath,amssymb,amsthm,amsfonts}
\usepackage[export]{adjustbox}
\usepackage{amsbsy}
\usepackage[utf8]{inputenc}
\usepackage[normalem]{ulem}
\usepackage{xspace}
\usepackage{cancel}
\usepackage[greek,english]{babel}
\usepackage{url}
\usepackage{wrapfig}
\usepackage{enumitem}
\usepackage{multicol}
\usepackage{moreenum}
\usepackage{xcolor}
\usepackage[pagewise]{lineno}
\newcounter{dummy}
\makeatletter
\newcommand\myitem[1][]{\item[#1]\refstepcounter{dummy}\def\@currentlabel{#1}}
\makeatother
\usepackage{subcaption}
\usepackage{comment}
\usepackage{tikz}
\usetikzlibrary{backgrounds}
\usepackage[bottom=2.5cm,top=2.5cm, left=2.5cm, right=2.5cm]{geometry}

\usepackage{multirow}
\usepackage{array,multirow}
\definecolor{LinkColor}{rgb}{0,0,0} 
\usepackage[colorlinks=true,linkcolor=LinkColor,citecolor=LinkColor,urlcolor= LinkColor, naturalnames, hyperindex, pdfstartview=FitH, bookmarksnumbered, plainpages]{hyperref}
\usepackage{cleveref}
\usepackage{tikz}
\usetikzlibrary{shapes, positioning}

\makeatletter
\newcommand{\slunlhd}{%
  \mathrel{\mathpalette\sl@unlhd\relax}%
}

\newcommand{\sl@unlhd}[2]{%
  \sbox\z@{$#1\lhd$}%
  \sbox\tw@{$#1\leqslant$}%
  \dimen@=\ht\tw@
  \advance\dimen@-\ht\z@
  \ifx#1\displaystyle
    \advance\dimen@ .2pt
  \else
    \ifx#1\textstyle
      \advance\dimen@ .2pt
    \fi
  \fi
  \ooalign{\raisebox{\dimen@}{$\m@th#1\lhd$}\cr$\m@th#1\leqslant$\cr}%
}
\makeatother

\newtheorem{theorem}{Theorem}[section]
\newtheorem{corollary}[theorem]{Corollary}

\newtheorem{proposition}[theorem]{Proposition}

\theoremstyle{definition}

\newtheorem{remark}[theorem]{Remark}
\newtheorem{example}[theorem]{Example}
\newtheorem{question}{Question}

\newcommand{\cut}{\textsf{cut}\xspace}

\newcommand{\Irr}{\operatorname{Irr}}
\newcommand{\Aut}{\operatorname{Aut}}

\newcommand{\Gal}{\operatorname{Gal}}
\newcommand{\Char}{\operatorname{Char}}

\newcommand{\U}{\textup{U}}

\title{Cut groups: the progress and the problems}

\author[S. Chahal]{Seema Chahal}
\address{(Seema Chahal) Department of Mathematics, Indian Institute of Technology Roorkee, Roorkee (Uttarakhand)-247667, India.}
\email{\href{mailto:seema_r@ma.iitr.ac.in}{seema\_r@ma.iitr.ac.in}}
\author[S. Maheshwary]{Sugandha Maheshwary}
\address{(Sugandha Maheshwary) Indian Institute of Technology Roorkee, Roorkee (Uttarakhand)-247667, India.}
\email{\href{mailto:msugandha@ma.iitr.ac.in}{msugandha@ma.iitr.ac.in}}

\thanks{The first author gratefully acknowledges the support by Science  \& Engineering Research Board(SERB),  DST (Department of Science and Technology), India (SRG/2023/000180).}


\date{}

\begin{document}
 
\maketitle

\begin{abstract}
This article focuses on the study of \cut groups, i.e., the groups which have only trivial  central units in their integral group ring. We provide state of art for \cut groups. The results are compiled in a systematic manner and have also been analogously studied for some generalised classes, such as quadratic rational groups and semi-rational groups etc. For these bigger classes, some results have been extended, while others have been posed as questions. The progress and the problems signify the development and the potential that holds in the topic of \cut groups and its generalisations.

\end{abstract}

\section{Introduction:  A story of cut groups} 
For a group $G$, let $\mathbb{Z}G$ denote its integral group ring. Over the years, people have investigated the unit group $\U(\mathbb{Z}G)$ immensely and in form of varied questions, for instance by studying its structure, by constructing its non-torsion units, by providing subgroups of finite index and their basis etc. Analogous questions have also been studied for $\mathcal{Z}(\U(\mathbb{Z}G))$, the centre of $\U(\mathbb{Z}G)$. The complete understanding of $\U(\mathbb{Z}G)$ is generally quite involving due to the presence of non-torsion units. Clearly, if $G$ is a finite group, then the trivial units of $\mathbb{Z}G$, i.e., elements of $\pm G$ are torsion units in $\U(\mathbb{Z}G)$. However, it rarely happens that these are the only units in $\U(\mathbb{Z}G)$, i.e., $\U(\mathbb{Z}G)=\pm G$. A complete classification of such groups was already provided by Higman in \cite{Hig40}, an article which supposedly initiated the rigorous study of $\U(\mathbb{Z}G)$. However, the parallel question for $\mathcal{Z}(\U(\mathbb{Z}G))$, i.e., the groups $G$ for which $\mathcal{Z}(\U(\mathbb{Z}G))=\pm\mathcal{Z}(G)$ has a different situation. Such groups are called as \cut groups, a term given in \cite{BMP17}, for the groups $G$ which have all \textbf{c}entral \textbf{u}nits in $\mathbb{Z}G$ \textbf{t}rivial (equivalently \textbf{t}orsion).

According to \cite{BBM20}, the study of \cut groups had already been in place in 1970's, when Z.\,F.\,Patay studied finite simple \cut groups and also classified alternating \cut groups (\cite{Pat75}, \cite{Pat78}). He also provided a description of nilpotent \cut groups of class $2$ (see \cite{AB91}, 3.10). It may be noted that Patay was a student of A. Bovdi and a study of \cut groups appears in Bovdi's book \cite{Bov87}. He proved that a finite group $G$ is \cut if and only if the character field $\mathbb{Q}(\chi)$ of each
complex irreducible character $\chi$ of $G$ is either $\mathbb{Q}$ or imaginary quadratic. In \cite{GP86} also, the need to study \cut groups has been pointed out. This was followed by work of Ritter and Sehgal \cite{RS90}, where they provided a group theoretic equivalent condition for a group to be a \cut group. The \cut groups also appears in \cite{Ale94}, where the author points out that the class of \cut groups properly contains the class  of rational groups. Recall that a group is said to be rational if the character field of each of its
complex irreducible character is $\mathbb{Q}$. It thus follows that all rational groups, for instance Symmetric groups, are indeed \cut groups.

It is apparent that the study of \cut groups has various perspectives. For, \cut groups may be studied from a group theoretic point of view, or via module theoretic interpretations, or purely in terms of character theory etc. Further, it follows from \cite{Bas64}, that a \cut group is a group whose $K_1(\mathbb{Z}G)$, the Whitehead group of $\mathbb{Z}G$, is finite. One may also relate the study of \cut groups to the normalizer property, as observed in \cite{JJdMR02}. A connection of \cut groups with certain fixed point properties, such as Kazhdan’s property or Serre’s property  is observed in \cite{BJJKT23}. Clearly, there is an interesting interplay of varied directions which make the study of \cut groups quite rich and involving (see also \cite{MP18}, Theorem 5).

In recent years, rigorous work on \cut groups has been done (\cite{BMP17}, \cite{Mah18}, \cite{Bac18}, \cite{MP18}, \cite{Bac19}, \cite{Tre19}, \cite{Gri20}, \cite{BCJM21}, \cite{Mor22}, \cite{BKMdR23}), thereby creating  several directions which are of independent interest too. 
In fact, the results on arbitrary (not necessarily finite) \cut groups also appear  in several works (\cite{JPS96}, \cite{MS99}, \cite{DMS05}, \cite{BMP19}). Recently, the so-called extended \cut groups, i.e., the groups whose non-torsion central units in the integral group ring form an infinite cyclic group, also have been studied (\cite{BBM20}, \cite{GKM25}). Other generalisations of this class happen to be the classes of quadratic rational groups and semi-rational groups, which have been rather studied as extensions of rational groups (\cite{CD10}, \cite{Ten12}, \cite{AD17}, \cite{dRV24}, \cite{Ver24}, \cite{PV25}).

The purpose of this article is to provide a systematic study of \cut groups and some of its generalisations. We survey the progress in this direction and highlight several relevant problems. In fact, the immediate questions and consequences of the known results have been addressed, yielding some new results. The article is structured as follows: In Section 2, we explore \cut groups via various directions, by studying several equivalent conditions for a group to be \cut. Section~3 is devoted to discussing the progress on \cut groups. This includes studying the properties borne by \cut groups and also the advancements made on this class. The section concludes by pointing on the classification of \cut groups in some well-known classes of finite groups, thereby indicating a huge range of examples for \cut groups. In Section 4, we study some classes of groups which generalise rational groups, and also \cut groups. For these classes of groups, we analyse the properties and results analogous to those provided for \cut groups in Section 3. Some new results for these generalised classes have been provided in this section and some natural questions have been put up. The article ends with closing remarks on \cut groups, shedding light on the potential of the topic.
\section*{Notation}
Throughout the article, we use standard notation. Unless stated otherwise, $G$ always denotes a finite group and its order is $|G|$. For an element $x\in G$, the order of $x$ is denoted by$|x|$ and $\langle x\rangle$ is the cyclic group generated by $x$. For $x,y\in G$, we write $x\sim y$, if $x$ is conjugate to $y$ in $G$. The set of elements conjugate to $x$ in $G$, i.e., the conjugacy class of $x$ is $C_x$. The centralizer, the normalizer and the automorphism group of $\langle x\rangle$  are respectively denoted by $C_G(\langle x\rangle)$, $N_G(\langle x\rangle)$ and $\Aut(\langle x\rangle)$. Let $\Irr(G)$ be the set of irreducible characters of $G$. We frequently use the following fields.
\begin{quotation}
	$\mathbb{Q}(\chi)$:=  $\mathbb{Q}(\chi(x):  x \in G),~ \mathrm{for} ~\chi \in \Irr(G)$\\
	$\mathbb{Q}(x) $:=~\,$\mathbb{Q}(\chi(x):  \chi \in \Irr(G)), ~\mathrm{for} ~x \in G$\\
	$\mathbb{Q}(G)$:=  $\mathbb{Q}(\chi(x): x\in G, \chi \in \Irr(G))$
	
\end{quotation}

\section{Cut groups: groups with only trivial central units in integral group rings}
In this section, we study various equivalent statements for a group to be \cut.\\

We begin with the constraints on the character fields of \cut groups, which rather serve as a characterisation for a group to be \cut.

\subsection{Character fields and simple components:}
In view of (\cite{Bov87} and {\cite{Hup67}, p.545), we have the following equivalent statements for a finite group $G$:
\begin{enumerate}
		\myitem[(E1)]\label{E1} $G$ is a \cut group, i.e., $\mathcal{Z}(\U(\mathbb{Z}G)) = \pm \mathcal{Z}(G).$
			
		\myitem[(E2)] If $\mathbb{Q}G \cong \bigoplus_{k=1}^{m} M_{n_k}(D_k)$ is the Wedderburn decomposition of the rational group algebra $\mathbb{Q}G$ where $m, n_k \in \mathbb{Z}_{\geq 0}$ and $D_k$ are division algebras, then for each $k$, $\mathcal{Z}(D_k) \cong \mathbb{Q}(\sqrt{-d})$ for some $d=d(k) \in \mathbb{Z}_{\geq 0}$. 
		\myitem[(E3)]\label{E3} For each $\chi \in \Irr(G), \mathbb{Q}(\chi)=\mathbb{Q}(\sqrt{-d})$ for some $d=d(\chi) \in \mathbb{Z}_{\geq 0}$.\\
		\end{enumerate}
	
	\begin{remark}
		It thus follows that if $G$ is a \cut group, then so are its homomorphic images.
	\end{remark} 
The next characterisation on \cut groups is based on its conjugacy classes. 
\subsection{Group theoretic criteria:} In \cite{RS90}, it has been proved that  a finite group $G$ is a \cut group if and only if the following holds:

\begin{enumerate}
\myitem[(E4)] For any $x \in G$, if $j\in \mathbb{N}$ is such that $ \gcd(j, |G|) = 1$, then $ x^j \sim x \text{ or } x^j \sim x^{-1}$.
	\end{enumerate}

\noindent The above criterion was refined in \cite{DMS05}, as follows:

\begin{enumerate}
	\myitem[(E5)]\label{E4} For any $x \in G$, if $j\in \mathbb{N}$ is such that $ \gcd(j, |x|) = 1$, then $ x^j \sim x \text{ or } x^j \sim x^{-1}$.\\
\end{enumerate}
Clearly, elements of order $1,2,3,4$ or $6$ always satisfy \ref{E4} and hence groups of exponent dividing $4$ or $6$ are always \cut. 

\begin{remark}\label{abelian}
	In fact, by observing that in abelian groups, each element forms its own conjugacy class, a finite abelian group is \cut, if and only if, it is of exponent $1,2,3,4$ or $6$.
\end{remark}

In an attempt to generalise the class of rational groups, Chillag et al. \cite{CD10} defined semi-rational and inverse semi-rational groups. An element $x\in G$ is said to be rational, if all generators of $\langle x \rangle$ form a single conjugacy class and a group is rational if all its elements are rational. Likewise, an element $x\in G$ is defined to be semi-rational, if all generators of $\langle x \rangle$ are contained in a union of at most two conjugacy classes. In particular, if all generators of $\langle x \rangle$ are contained in conjugacy classes of $x$ or $x^{-1}$, then $x$ is called inverse semi-rational. The group $G$ is called semi-rational (respectively inverse semi-rational), if all elements of $G$ are semi-rational (respectively inverse semi-rational).  

Further, by identifying the quotient group $B_G(x)=N_G(\langle x\rangle)/C_G(\langle x \rangle)$ as a subgroup of $Aut(\langle x \rangle)$, they provided an equivalent criterion for an element to be inverse semi-rational. Moreover, it is  proved that $x$ is semi-rational if and only if $\mathbb{Q}(x)=\mathbb{Q}(\chi(x)~|~\chi \in \Irr(G))$ is either quadratic or $\mathbb{Q}$ (\cite{CD10}, \cite{Ten12}), using which it is deduced in \cite{BCJM21} that $x$ is inverse semi-rational, if and only if, $\mathbb{Q}(x)$ is either $\mathbb{Q}$ or imaginary quadratic.

 It is apparent from \ref{E4} that the class of inverse semi-rational groups is exactly the class of \cut groups. Consequently, we have the following equivalent statements:

\begin{enumerate}
	\myitem[(E6)] $G$ is an inverse semi-rational group.
	
	\myitem[(E7)]\label{E7} For each $x \in G$, $Aut(\langle x \rangle) = B_G(x)\langle \tau \rangle$ where $\tau \in Aut(\langle x \rangle)$ is the inversion automorphism.
	
	\myitem[(E8)]\label{E8}For each $x \in G$, $\mathbb{Q}(x)=\mathbb{Q}(\sqrt{-d})$ for some $d=d(\chi) \in \mathbb{Z}_{\geq 0}$.\\
	
\end{enumerate}

Note: \ref{E3} and \ref{E8} essentially state that a group $G$ is \cut if and only if for all rows (columns) of the character table of $G$, the field extension of $\mathbb{Q}$ generated by the entries of that row (columns) equals $\mathbb{Q}$ or an imaginary quadratic extension.

\newpage
\begin{remark}It may be pointed out that following are equivalent statements for a rational group $G$.
	\begin{enumerate}
		\item[(i)] All elements of $G$ are rational, i.e., for any $x\in G$ all generators of $\langle x \rangle$ form a single conjugacy class. 
		\item[(ii)] For every $x \in G$, $\mathbb{Q}(x)=\mathbb{Q}$.
		\item[(iii)] For every $x \in G$,  $\Aut(\langle x \rangle) \cong{B}_G(x)$.
		\item[(iv)]  All irreducible characters of $G$ are rational, i.e., $\mathbb{Q}(\chi)=\mathbb{Q}$, for every $\chi \in \Irr(G)$.
	\end{enumerate}
\end{remark}

\begin{remark}\label{PrimitiveRoots}
	If $x \in G$ is such that the order of $|x|$ has a primitive root, say $r$, then the criterion in \ref{E4} can be refined, and becomes quite effective for computational purposes. For, if $|x|$ has a primitive root, then $|x|=2,4,p^k$ or $2p^k$, where $p$ is an odd prime and $k \in \mathbb{N}$. As already pointed, if $|x|=2$ or $4$, then $x$ is inverse semi-rational. Now, if $|x|=p^k$ or $2p^k$, then $x$ is inverse semi-rational if and only if $ x^r \sim  x$ or $ x^r \sim  x^{-1}$. In particular, if $G$ is an odd order group, then $G$ is \cut, if and only if, for every $x\in G$, we have that $ x^5 \sim  x^{-1}$ and $|x|$ is either $7$ or a power of $3$ (\cite{Mah18}, Theorem 1). Additionally, a $2$-group $G$ is \cut if and only if every element $x\in G$ satisfies $x^3\sim x$ or $x^{-1}$ (\cite{BMP17}, Theorem 3).
\end{remark}

\begin{remark}
	A group $G$ is said to have the Magnus property, or briefly an MP group, if whenever two elements $x$ and $y$ have the same normal closure, then they are either conjugates or inverse conjugates, i.e., $x\sim y$ or $x\sim y^{-1}$. Clearly, a finite MP group is indeed a \cut group. Recently, finite groups with Magnus property have been explored \cite{GM24}. It follows that the class of \cut groups is strictly larger than those of MP groups.	This is because, by (\cite{GM24}, Proposition 1.1), a finite MP group is solvable, whereas $S_n$, $n\geq 5$ are non-solvable \cut groups. However, it will be interesting to study the properties of one class relative to the other.

\end{remark}


\subsection{Rank zero criterion:} It is well known that the torsion central units of $\mathbb{Z}G$ are precisely $\pm x$, where $x\in \mathcal{Z}(G)$. The non-torsion central units of $\mathbb{Z}G$ form a free abelian group of finite rank. Let $\rho(G)$ denote the free rank of $\mathcal{Z}(\U(\mathbb{Z}G))$. Clearly, $G$ is \cut if and only if $\rho(G)=0$. The rank $\rho(G)$, in terms of the conjugacy classes of $G$, have been worked upon in \cite{Fer04}. To state  these results, we first recall $\mathbb{F}$-classes in a group $G$, where $\mathbb{F}$ is a field such that $\Char(\mathbb{F})$ does not divide $|G|$. Let $e$ be the exponent of $G$ and let $\xi$ be a primitive  $e^{\mathrm {th}}$ root of unity. For $x\in G$, $C_x^{\mathbb{F}} = \cup_{\sigma \in \Gal(\mathbb{F}(\xi)/\mathbb{F})}C_{x^{\sigma}}$ is known as $\mathbb{F}$-class of $x$. In particular, $\mathbb{R}$-class of $x$ is $C_x^{\mathbb{R}}$ which equals $C_x \cup C_{x^{-1}}$, and the $\mathbb{Q}$-class of $x$ is $C_x^{\mathbb{Q}}$ which is $\cup_iC_{x^i}$, where $i \in \{1 \leq i \leq e | \gcd(i, e)=1\}$. Ferraz \cite{Fer04} proved that the free rank of $\mathcal{Z}(\U(\mathbb{Z}G))$ is the difference in the number of $\mathbb{R}$-classes and $\mathbb{Q}$-classes in $G$. Consequently, we have the following equivalences for a \cut group $G$:
\begin{enumerate}
\myitem[(E9)] $\rho(G)=0$.
\myitem[(E10)] The number of $\mathbb{R}$-classes in $G$ is same as that of its $\mathbb{Q}$-classes.
	\myitem[(E11)] The number of simple components in the Wedderburn decomposition of $\mathbb{R}G$ and $\mathbb{Q}G$ is equal.
	\myitem[(E12)] The number of real irreducible characters of a group $G$ is same as that of rational irreducible characters.
\end{enumerate}
	
Similar rank formulae also appear in the work of \cite{Pat78} (c.f. \cite{BBM20}) and \cite{RS05} and a generalisation is provided in \cite{BBM20} (see Proposition 2.2). It may be noted that computing the rank $\rho(G)$ of a group $G$ has been a topic of independent interest, as can be observed in \cite{AA69}, \cite{AKS08}, \cite{JOdRV13} and \cite{BK22} etc.

\subsection{Irreducibility over extensions:} Let $K_1(\mathbb{Z}G)$ denote the Whitehead group of $\mathbb{Z}G$ and for an ideal $\mathfrak{a}$ of $\mathbb{Z}G$, let 
$K_1(\mathbb{Z}G, \mathfrak{a})$ denote the relative group. Bass proved that for a finite group $G$, $K_1(\mathbb{Z}G)$ is a finitely generated abelian group of rank $r-q$, where $r$ and $q$ denote the number of real and rational irreducible characters of $G$ respectively (\cite{Bas64}, Corollary 20.3). We have the following equivalent statements on \cut groups from his work.
\begin{enumerate}
		\myitem[(E13)] $K_1(\mathbb{Z}G)$ is finite.
		\myitem[(E14)]  $K_1(\mathbb{Z}G, \mathfrak{a}) $ is finite for all ideals $\mathfrak{a}$ of $\mathbb{Z}G$.
		\myitem[(E15)]\label{E15} Every irreducible $\mathbb{Q}G$-module remains irreducible under scalar extension from $\mathbb{Q}$ to $\mathbb{R}$.
\end{enumerate}

In terms of \cite{dRV24}, a group $G$ is said to be inert over a subfield $\mathbb{F}$  of $\mathbb{C}$, if for every irreducible $\mathbb{Q}G$-module $M$, the $\mathbb{F}G$-module $\mathbb{F}\otimes _\mathbb{Q} M$ is irreducible. In other words, \ref{E15} states that $G$ is inert over $\mathbb{R}$. Hence, from the results of \cite{dRV24}, we obtain the following equivalent statements for a \cut group $G$. We use the following notation: $G$ is group of exponent $e$ and $\xi$ is a primitive $e^{\mathrm{th}}$ root of unity. Set $\mathbb{Q}_G=\mathbb{Q}(\xi)$ and $\mathbb{F_\xi}=\mathbb{Q}(\xi +\xi^{-1})$. For a subfield $\mathbb{K}$ of $\mathbb{C}$ and for $\chi \in Irr(G)$, $\chi_\mathbb{K}=\{\chi^\sigma~|~\sigma\in \mathcal{U}_{\mathbb{K}}\}$, where  $\mathcal{U}_{\mathbb{K}}=\{r~|~\gcd(r,e)=1 ~\mathrm{and}~\sigma_r(x)=x~\mathrm{for~every~}x\in \mathbb{K}\cap \mathbb{Q}_G\}$ }.

\begin{enumerate}
	\myitem[(E16)] $G$ is inert over $\mathbb{R}$.
	\myitem[(E17)] $G$ is inert over $\mathbb{F}_\xi$.
	\myitem[(E18)] $C_x^{\mathbb{F}_\xi}=C_x^{\mathbb{Q}}$.
	\myitem[(E19)]   $\chi_{\mathbb{F_\xi}} = \chi_{\mathbb{Q}}$, for every $\chi \in \Irr(G)$.
	\myitem[(E20)] 	$\mathbb{Q}(\chi) \cap \mathbb{F_\xi}= \mathbb{Q}$ for every $\chi \in \text{Irr}(G)$.
	\myitem[(E21)] $\mathbb{Q}(x) \cap \mathbb{F_\xi}= \mathbb{Q}$ for every $x \in G$.
	 \myitem[(E22)] For every $\sigma \in \Gal(\mathbb{Q}_G/\mathbb{Q})$ and  for every $\chi \in \Irr(G) $, there exists $\tau_{\chi} \in  \Gal(\mathbb{Q}_G/ \mathbb{F_\xi})$  such that $\chi^{\sigma}=\chi^{\tau_{\chi}}.$
	\myitem[(E23)] For every  $\sigma \in \text{Gal}(\mathbb{Q}_G / \mathbb{Q})$ and  for every $x \in G$, there exists $\rho_x \in \Gal(\mathbb{Q}_G/ \mathbb{F_\xi})$ such that 
	$ \chi^\sigma(x) = \chi^{\rho_x}(x) $ for every $\chi \in \text{Irr}(G).$
	\myitem[(E24)] The restriction map 
	$\Gal(\mathbb{Q}_G/ \mathbb{F_\xi}) \to {\Gal}(\mathbb{Q}(\chi)/\mathbb{Q})$
	is surjective for every $\chi \in \text{Irr}(G)$.
	\myitem[(E25)] The restriction map $\Gal(\mathbb{Q}_G/ \mathbb{F_\xi})\to {\Gal}(\mathbb{Q}(x)/\mathbb{Q})$
	is surjective for every $x \in G$.
	
\end{enumerate}

\subsection{The normalizer criterion:} For a group $G$, set $\U:= \U(\mathbb{Z}G)$ and let $\mathcal{N}_{{\U}}(G)$ be the normalizer of $G$ in $\U$. Obviously, $G. \mathcal{Z}(\U)\subseteq \mathcal{N}_{\U}(G)$. However, if $G. \mathcal{Z}(\U)= \mathcal{N}_{\U}(G)$, then $G$ is said to have the normalizer property. The study of the normalizer property has been done for several classes of groups, especially due to its influence on the Isomorphism problem. For instance, it is known to hold for finite nilpotent groups and for groups of odd order. As pointed out in (\cite{BMP17}, Remark 4), it follows from the work of \cite{Maz99} and \cite{LH12} that if a finite group $G$ has the \cut property, then $G$ has the normalizer property and the converse is not true in general. In fact, the result holds good for an arbitrary (not necessarily finite) \cut group as well, as can be seen from the results of \cite{JJdMR02}. We close this section by stating yet another equivalent statement for a group $G$ to be \cut due to Jespers et al. (\cite{JJdMR02}, Corollary 1.7).

 \begin{enumerate}
	\myitem[(E26)]\label{E26} $\mathcal{N}_{{\U}}(G)=\pm G.$ 
	\end{enumerate}

	{Note: The equivalent statements \ref{E1}-\ref{E26} for a finite \cut group stated in this section may be seen as refinement of Theorem 5 of \cite{MP18}.}

\section{Progress and problems on cut groups}
As the study of \cut groups has been progressively done in last few years, a set of interesting results have turned up. We have already observed some fundamental properties of \cut groups. For instance, groups of exponent dividing $4$ or $6$ are always \cut; the homomorphic images of \cut groups are \cut. In this section, we discuss the properties borne by \cut groups and progress made about understanding this class of groups.


\subsection{Subgroups of \cut groups:} A subgroup of a \cut group need not essentially be a \cut group. For, otherwise Cayley's theorem would yield all finite groups to be \cut, as Symmetric groups are \cut (in fact, rational). Being a \cut group is not even closed for normal subgroups. This is again observable from the fact that not all alternating groups are \cut (\cite{BCJM21}, Theorem 5.1). However, it easily follows from the definition that the centre of a \cut group is indeed a \cut group. It is of natural interest to know whether Sylow subgroups of \cut groups are \cut. Firstly, observe that if a $p$-group has \cut property, then all the quotients in its upper central series also have \cut property. Consequently, $p=2$ or $3$, as also observed in (\cite{BMP17}, Theorem 1). So, the question reduces to checking whether a Sylow $p$-subgroup of a \cut group is again a \cut group or not, where $p\in \{2,3\}$. In \cite{BCJM21}, it has been shown that Sylow $2$-subgroup of a \cut group need not be \cut but it is not known whether a Sylow $3$-subgroup of a \cut group is \cut. However, for Sylow $3$-subgroups, a partial positive answer is given for some substantial cases. We observe that for these cases, the result holds good even for Sylow $2$-subgroups, except possibly for Monster group.

\begin{theorem}\label{Sylow p-subgroup}
	Let $G$ be a \cut group and let $P$ be its Sylow $p$-subgroup, where $p\in \{2,3\}$. Then $P$ is also a \cut group provided if:
	\begin{enumerate}
		\item[\rm{(i)}] $P$ is an abelian group or a normal subgroup of $G$.
		\item[\rm{(ii)}] $G$ is a supersolvable group.
		\item[\rm{(iii)}]$G$ is a Frobenius group. 
		\item[\rm{(iv)}]$G$ is a simple group, except possibly when $p=2$ and $G$ is the Monster group. 
		\item[\rm{(v)}] $G$ is a group of odd order.
		\item[\rm{(vi)}]$G$ is a symmetric group.
	
	\end{enumerate}
\end{theorem}
The proof of above theorem essentially follows from \cite{Gri20} and \cite{BCJM21} along with the techniques contained therein. Since most of the results given in these articles only address  the case when $p=3$, we restate the proofs making them inclusive for the case when $p=2$, thereby also improving accessibility for the reader.
\begin{proof}Let $G$ be a \cut group and let $P$ be its Sylow $p$-subgroup, where $p\in \{2,3\}$.   \begin{enumerate}
	\item[(i)] Suppose $x\in P$ is an element of order $p^f$, where $f\geq 1$. We need to show that $x$ is inverse semi-rational in $P$. If $f=1$, then $x$ is clearly an inverse semi-rational element in $P$. Hence, assume $f\geq 2$. Since, $x$ is inverse semi-rational in $G$, by \ref{E7}, we have that $Aut(\langle x \rangle) = B_G(x)\langle \tau \rangle$ where $\tau \in Aut(\langle x \rangle)$ is the inversion automorphism. 

Suppose firstly that $P$ is abelian, so that $P\subseteq C_G(\langle x \rangle)$, and consequently $p\nmid |B_G(x)|$. Recall that $\Aut(\langle x\rangle) \cong C_{p^{f-1}}\times C_{p-1}$, if $p\neq 2$, and $\Aut(\langle x\rangle) \cong C_{2^{f-2}}\times C_{2}$, if $p=2$. As $[Aut(\langle x \rangle) : B_G(x)]\leq 2$, it follows that $p^f=4$ and hence $P$ is a \cut group, being a group of exponent $2,3$ or $4$.

Now, let $P$ be a normal subgroup of $G$, i.e., it is a unique Sylow $p$-subgroup of $G$. For $x\in P$, it follows from the structure of $\Aut(\langle x\rangle)$ and \ref{E7}, that $B_G(x)$ is a $p$-group. As $x$ is inverse semi-rational element in $G$, it follows that for every $1\leq i \leq p^f$ such that $\gcd(i,p)=1$ there exists an element $g\in G$, such that $gx^jg^{-1}=x$ or $x^{-1}$. In particular, $g\in N_G(\langle x \rangle)$. Since, $p'$-part of $g$ belongs to $C_G(\langle x \rangle)$, replacing $g$ by its $p$-part, if necessary, assume that $g$ is also a $p$-element. Hence, $g$ belongs to a Sylow $p$-subgroup of $G$. But in this case, Sylow $p$-subgroup of $G$ is unique. Consequently, $g\in P$, i.e., $x$ is inverse semi-rational in $P$.

\item[(ii)] If $G$ is a supersolvable group, then $G$ has a Sylow tower (\cite{Hup67}, Satz. 9.1), and hence its $p$-length is at most $1$, for any prime $p$ dividing the order of $G$. We thus have that $G$ has normal subgroups $M$ and $N$ such that $1\subseteq M \subseteq N \subseteq G$ where $M$ and $G/N$ are $p'$-groups and $N/M$ is a $p$-group. The proof of (\cite{BCJM21}, Proposition 6.3) now works for $p\in\{2,3\}$ and we have that $P$ is \cut in this case.

\item[(iii)] If $G$ is a Frobenius \cut group, then let $K$ and $H$ respectively be the Frobenius kernel and Frobenius complement of $G$. As $\gcd(|K|, |H|)=1$, a Sylow $p$-subgroup $P$ of $G$ is isomorphic to Sylow $p$-subgroup of either $K$ or $H$. If $p$ divides  $|K|$, then $P$ is normal in $G$ because $K$ is nilpotent and characteristic in $G$. Hence, $P$ is \cut by (i). Now, if $p$ does not divide $|K|$, then $p$ divides $|H|$ and $P$ is a Sylow $p$-subgroup of $H$, so that by (\cite{Isa08}, Corollary 6.17), it follows that $P$ is either a cyclic or a quaternion group. If $P$ is cyclic, then it is \cut by (i), and if it is quaternion, it is indeed a \cut group being of exponent $4$.

 \item[(iv)] If $G$ is a simple \cut group, then it is one of the groups listed in (\cite{BCJM21}, Theorem 5.1). Using \texttt{GAP} \cite{GAP4}, it has been verified that for $p\in \{2,3\}$, Sylow $p$-subgroups of these groups are again \cut, except possibly Sylow $2$-subgroup of Monster group $M$.

\item[(v)-(vi)] If $G$ is a Symmetric group or a group of odd order, then the result has been observed in (\cite{CD10}, Remark 14) and \cite{Gri20} respectively.

\end{enumerate}
\end{proof}

Note that there exist \cut groups for which Sylow $2$-subgroups are not \cut. For instance, Sylow $2$-subgroups of \texttt{SmallGroup}\texttt{[384,18033]} and \texttt{SmallGroup[384,18040]} are not \cut \cite{BCJM21}. However, for Sylow $3$-subgroups, no counter example has been found even on considering all \cut groups of order at most $2000$, as well as those of order $2^2.3^6$, $2^3.3^6$ and $2^2.3^7$. Thus, as also pointed in \cite{Bac19}, it is meaningful to state the following question.

\begin{question}
Is Sylow $3$-subgroup of a $\cut$ group, again a \cut group?
\end{question}

\subsection{Direct products of \cut groups:} 
In general, the direct products of \cut groups need not be \cut. For instance, $C_3$ and $C_4$ are \cut groups, whereas their direct product $C_{12}$ fails to be a \cut group because all four generators of this group belong to different conjugacy classes. However, for the direct product of two \cut groups to be \cut, following criterion was provided in (\cite{BKMdR23}, Lemma 2.7(7)).
\begin{proposition} Let $G_1$ and $G_2$ be \cut groups.	The direct product  $G_1\times G_2$ is \cut if and only if for every two integers $n_1$ and $n_2$ such that $G_1$ has a non-rational element of order $n_1$ and $G_2$ has a non-rational element of order $n_2$, one has $\gcd(n_1, n_2)\in \{3, 4, 6\}$.
\end{proposition}

We provide another criterion for the direct product of two \cut groups to be \cut in terms of character fields.

\begin{theorem}\label{cut_direct}
Let $G_1$ and $G_2$ be \cut groups.	The direct product $G_1\times G_2$ is \cut, if and only if, one of the following holds:
	\begin{enumerate}
		\item[\rm{(i)}] either $G_1$ or $G_2$ is rational.
		\item[\rm{(ii)}] $\mathbb{Q}(G_1)=\mathbb{Q}(G_2)=\mathbb{Q}(\sqrt{-d})$, for some $d\in \mathbb{N}$.
	\end{enumerate}
\end{theorem}

\begin{proof}
Suppose first that $G_1\times G_2$ is a \cut group. Then, it follows that both $G_1$ and $G_2$ are \cut. If none of $G_1$ or $G_2$ is rational, then there exist $\chi_1\in \Irr(G_1)$ and $\chi_2\in \Irr(G_2)$ such that $\mathbb{Q}(\chi_1)=\mathbb{Q}(\sqrt{-d_1})$ and $\mathbb{Q}(\chi_2)=\mathbb{Q}(\sqrt{-d_2})$ for some $d_1,d_2\in \mathbb{N}.$ Since $\chi_1\times \chi_2\in \Irr(G_1\times G_2)$, we have that $\mathbb{Q}(\sqrt{-d_1},\sqrt{-d_2})\subseteq \mathbb{Q}(\chi_1\times \chi_2)$, but $G_1\times G_2$ being a \cut group implies that $\mathbb{Q}(\chi_1\times \chi_2)$ must be $\mathbb{Q}(\sqrt{-d})$ for some $d\in \mathbb{N}.$ Therefore, we must have $d_1=d_2=d$. Further, if $\chi_1'$ is any character of $G_1$ such that $\mathbb{Q}(\chi_1')=\mathbb{Q}(\sqrt{-{d_1}'})$, ${d_1}'\in \mathbb{N}$, then ${d_1}'=d.$ Thus, $\mathbb{Q}G_1=\mathbb{Q}(\sqrt{-d})$, and similarly $\mathbb{Q}G_2=\mathbb{Q}(\sqrt{-d})$. Consequently, $\mathbb{Q}(G_1)=\mathbb{Q}(G_2)=\mathbb{Q}(\sqrt{-d})=\mathbb{Q}(G_1 \times G_2).$ Clearly, the converse follows.
\end{proof}
\newpage
Inductively, we obtain the following corollary from \Cref{cut_direct}. 
\begin{corollary} The direct product $G_1\times G_2\times \ldots \times G_n$ of finite \cut groups $G_k$, where ${k\in \{1,2,...,n\}}$ is \cut if and only if for all non-rational groups $G_i$, $i\in I \subseteq \{1,2,...,n\}$, we have that
  $\mathbb{Q}(G_i)=\mathbb{Q}(\sqrt{-d})$, for a fixed natural number $d$.
\end{corollary}

In particular, if $G$ is \cut, then so is $G\times C_2^n$, $n\in \mathbb{N}$.
\begin{remark}It may be noted that the \cut property is direct product closed for groups of odd order (\cite{MP18}, Corollary 5), in particular for $3$-groups. Though the \cut property is not direct product closed for $2$-groups (\cite{BMP17},  Remark 1),  it is direct product closed for $2$-groups of class at most $2$ (\cite{Mah18}, Corollary 3). 
\end{remark}

\subsection{Rational conjugacy classes:}
In a group $G$, a conjugacy class $C_x$ is called rational, if $x$ is a rational element in $G$, i.e., $\mathbb{Q}(x)=\mathbb{Q}$ and a character $\chi$ of $G$ is called rational, if $\mathbb{Q}(\chi)=\mathbb{Q}.$ 
In general, the number of rational conjugacy classes and rational irreducible characters of a group $G$ are not the same. These number coincide if the group has all Sylow subgroups abelian \cite{Bro71} or if 
the Sylow $2$-subgroups are cyclic \cite{NT10}. The numbers are also same when the group has at most two rational irreducible characters \cite{NT08}. Similar work can also be found in \cite{Sou25}. It has been proved that for \cut groups also, these numbers are same.
\begin{theorem}(\cite{BCJM21}, Proposition 3.2) For a finite \cut group $G$, the number of rational irreducible characters of $G$ is equal to the number of rational conjugacy classes of $G$.
\end{theorem}

It has been proved in \cite{Nav09}, that in a group of odd order, the number of irreducible quadratic characters is same as that of quadratic conjugacy classes, where an irreducible character $\chi$ is called quadratic if $[\mathbb{Q}(\chi): \mathbb{Q}]=2$ and a conjugacy class $C_x$ is called quadratic if $[\mathbb{Q}(x):\mathbb{Q}]=2$. In view of above theorem, this result also holds good for \cut groups.

\begin{corollary}
	In a finite group $G$, the number of irreducible quadratic characters of $G$ is equal to the number of quadratic conjugacy classes in $G$.
\end{corollary}

We state examples from \cite{BCJM21}, (see also \cite{Ten12}, Section 6) of groups for which the number of rational irreducible characters is not equal to the number of rational conjugacy classes. Also, for each of these groups, the number of irreducible quadratic characters is not same as that of its quadratic conjugacy classes.\\

\begin{example}\label{G1}
	Consider $G_1= \langle a, b, c \;|\; a^2=b^2=c^8=1, b^c=b, b^a=bc^4, c^a=c^3\rangle$, which is not a \cut group but a semi-rational group. In $G_1$, there are $10$ irreducible rational characters but $8$ rational conjugacy classes. Also, there are $6$ quadratic conjugacy classes but no quadratic irreducible characters in $G_1$.
\end{example}

\begin{example}\label{G2}
	Similarly, if $G_2=\langle a, b, c  \;|\; a^2=b^2=c^8=1, b^c=b, b^a=b, c^a=bc^3 \rangle$, then $G_2$ has $6$ irreducible rational characters but $8$ rational conjugacy classes. Further, $G_2$ has $2$ quadratic conjugacy classes but $8$ quadratic irreducible characters. 
\end{example}

Note that the groups $G_1$ and $G_2$ are groups of least order for which the number of  irreducible rational characters is different from that of rational conjugacy classes. However, the number of quadratic conjugacy classes and quadratic irreducible characters is different even in $C_{24}$.

 \subsection{Prime spectrum of \cut groups:}

 We observe that the property of $G$ being  \cut has strong bearing on the primes dividing the order of $G$. Let the prime spectrum of the group $G$, i.e., the set of primes dividing $|G|$ be denoted by $\pi(G)$. 
 
 Firstly, if $G$ is a \cut group then, we must have either $2\in \pi(G)$ or $3\in \pi(G)$. For, if $2\not \in \pi(G)$, then $G$ being odd order group is solvable. Considering the abelian quotient group $G/G'$, where $G'$ is the commutator subgroup of $G$, we see that exponent of $G/G'$ must be $3$, implying that $3\in \pi(G)$. If $G$ has the \cut-property, then so does $G/\mathcal{Z}(G)$. Now, if $G$ is nilpotent group, then each consecutive quotient in the upper central series of $G$, being abelian \cut group, is of exponent $1, 2, 3, 4$ or $6$ and hence $\pi{(G)}\subseteq\{2,3\}$. Such a constraint is also observed for odd order groups, and even for solvable groups. We put together the results on $\pi(G)$ from \cite{BMP17}, \cite{Mah18} and \cite{Bac18}.
 
 \begin{theorem}\label{prime_spectrum_cut} If $G$ is a \cut group, then
 	\begin{enumerate}
 		\item[\rm{(i)}] Either $2\in \pi(G)$ or $3\in \pi(G)$.
 		\item[\rm{(ii)}] If $G$ is nilpotent, then $\pi{(G)}\subseteq\{2,3\}$. 
 		\item[\rm{(iii)}] If $G$ is odd order group, then $\pi{(G)}\subseteq\{3,7\}$.
 		\item[\rm{(iv)}] If $G$ is solvable, then $\pi{(G)}\subseteq\{2,3, 5,7\}$.
 		
 	\end{enumerate}  
 \end{theorem}

If a $p$-group is \cut then, $p=2$ or $3$ and it also follows that the quotients in the upper and lower central series have exponent at most $4$ \cite{BMP17}. Consequently, for a finite nilpotent \cut group $G= G_2 \times G_3$, where $G_2$ and $G_3$ are  respectively the Sylow $2$ and Sylow $3$ subgroups of $G$, we obtain that the exponents of the quotients in the upper and lower central series of	$G$ are at most $12$. \\

As remarked in \cite{BCJM21}, every $\{2,3\}$-group can be embedded in a rational $\{2,3\}$-group and every nilpotent $\{2,3\}$-group can be embedded in a nilpotent $\{2,3\}$-\cut (but not rational) group. Also, because Sylow $5$-subgroup of a solvable rational group is elementary abelian, not every $\{2,3,5\}$-group can be embedded in a solvable rational group. However, it has been proved that each $5$-group as well as a $7$-group can be embedded into a solvable \cut group. This naturally raises the following question.

\begin{question}\label{embedding_cut}
Can every solvable $\{2,3,5,7\}$-group be embedded into a solvable cut group?
\end{question}

Clearly, the answer is positive for nilpotent $\{2,3,5\}$-groups or nilpotent $\{2,3,7\}$-groups, but one may look into weaker versions of this question as well.\\

We now see how \cut property classifies unit groups of integral group rings as per their central heights.
\subsection{The upper central series of $\U(\mathbb{Z}G)$:}

For a group $G$, let $\U=\U(\mathbb{Z}G)$ and let  $$⟨1⟩ = \mathcal{Z}_0(\U)\subseteq \mathcal{Z}_1(\U) \subseteq...\subseteq\mathcal{Z}_n(\U)\subseteq\mathcal{Z}_{n+1}(\U)... $$ 
be the upper central series of $\U$. The central height of $\U$ is the smallest integer $n\geq 0$ such that $\mathcal{Z}_n(\U)=\mathcal{Z}_{n+1}(\U)$. It turns out that the upper central series of $\U$ is  completely determined by $\mathcal\U$ and group structure of $G$. It is not wrong to say that the \cut property helps to classify the groups $G$ according to the central heights of $\U(\mathbb{Z}G)$. To see this, note that the central height of $\U$ is known to be at most $2$ and equals $2$ if, and only if, $G$ is a $Q^{*}$-group (\cite{AHP93},\cite{AP93}). By a $Q^{*}$-group, we mean a group $G$ which has an element $a$ of order $4$ and an abelian subgroup $H$ of index $2$, which is not an elementary abelian $2$-group, and is such that $G = \langle H, a\rangle, a^{-1}ha=h^{-1}$, for all $h\in H$ and $a^2 =b^2$, for some $b\in H$. If $G$ is a $Q^{*}$-group, i.e., if the central height of $\U$ is $2$, then the second centre is known in terms of the centre of $\U$. More precisely, $\mathcal{Z}_2(\U)=T\mathcal{Z}_1(\U)$, where $T=\langle b\rangle \oplus E_2$, $E_2$ being an elementary abelian $2$-group. Therefore, if $G$ is not a $Q^{*}$-group, then the central height must be $0$ or $1$. Clearly, a group has central height $0$, essentially translates to saying that $G$ is a \cut group with trivial centre. In all other cases, the central height of $\U$ equals $1$. Hence, the understanding of the upper central series of $\U$ comes down to the study of $\mathcal{Z}(\U)$, the centre of $\U$.\\

In this direction, we have the following result for odd order groups.
\begin{proposition}
Let $G$ be an odd-order group. The central height of $\U(\mathbb{Z}G)$ is $1$, except when $G$ is a \cut group containing an element of order $7$, and in this case, the central height of $\U(\mathbb{Z}G)$ is $0$.
\end{proposition}

We next compile the results based on more developments about \cut groups.

\subsection{Composition factors of \cut groups:}It has been of interest to determine the simple groups which may appear as composition factors of rational groups. For solvable rational groups, Gow \cite{Gow76} showed that the possible composition factors are cyclic groups of order at most $5$ and for non-solvable rational groups, Feit and Seitz \cite{FS88} proved that only finitely many groups of Lie type and alternating groups may occur as non-abelian composition factors. In this direction, Trefethen provided a complete list of non-abelian composition factors of finite \cut-groups.

\begin{theorem}\label{CompositionFactors}(\cite{Tre19}, Theorem 1.1) Let $G$ be a finite \cut group. The  non-abelian simple groups which can occur as a composition factor of $G$ are precisely amongst the following:
		\begin{enumerate}
		\item[\rm{(i)}] Alternating groups: $A_n$, $n \geq 5$.
		\item[\rm{(ii)}] Sporadic groups:	$M_{11}, M_{12}, M_{22}, M_{23}, M_{24}, \text{Co}_1, \text{Co}_2, \text{Co}_3, \text{HS}, \text{McL}, \text{Th}, M.$
		\item[\rm{(iii)}] Groups of Lie type: $L_2(7), L_2(8), L_3(4), U_3(3), U_3(4), U_3(5), U_3(8), U_4(2), U_4(3),
		U_5(2),\\ U_6(2), S_6(2), O_8^+(2), G_2(4), {}^2F_4(2)', {}^3D_4(2) $.
	\end{enumerate}
\end{theorem} 

\subsection{Degrees of character fields in \cut groups:}
Following \Cref{CompositionFactors}, Moreto \cite{Mor22}, provided a bound for $\mathbb{Q}(G)$, when $G$ is a \cut group.

\begin{theorem}\label{degree bound}(\cite{Mor22}, Theorem A)
	Let $G$ be a \cut group and let $A_t$ be the largest alternating group that appears as a composition factor of $G$. Then 
	$$ [\mathbb{Q}(G) : \mathbb{Q}] \leq 2^{\max\{54,~ p_t + 1\}},$$
	where $p_t$ is the number of primes less than or equal to $t$.
\end{theorem}

Note that for any rational group $G$, $[\mathbb{Q}(G):\mathbb{Q}]$ is simply $1$. In \cite{Bac19}, Bachle asked if there is a class $\mathfrak{C}$ of groups containing rational groups for which a uniform bound for $[\mathbb{Q}(G):\mathbb{Q}]$ may be given for every $G\in \mathfrak{C}$. This tempts whether the class of \cut groups has such a property. We restate this question from \cite{Bac19}.
\begin{question}\label{degree_question}
	Does there exist a $c$ such that $[\mathbb{Q}(G):\mathbb{Q}]\leq c$ for every \cut group $G$?
\end{question}

\Cref{degree bound} was motivated by \Cref{degree_question} and \Cref{CompositionFactors}. It states that if \cut groups do not contain arbitrarily large Alternating groups as composition factors, then the \Cref{degree_question} has affirmative answer, i.e., if \Cref{degree_question} has a negative answer, then counter examples will only come from family of groups with arbitrarily large Alternating groups as composition factors.\\

Note that for solvable \cut groups, $[\mathbb{Q}(G):\mathbb{Q}]\leq 2^5$ \cite{Ten12}.

\subsection{ The Gruenberg-Kegel graphs:} The Gruenberg-Kegel graph of a finite group $G$ is the undirected graph whose set of vertices is the prime spectrum of $G$, and two distinct vertices $p$ and $q$ are joined by an edge, if and only if $G$ contains an element of order $pq$. The Gruenberg-Kegel graph of a group $G$, which is also called as the prime graph of $G$, is generally abbreviated as $GK$-graph of $G$ and is denoted by $\Gamma_{GK}(G)$.  The $GK$-graph of $G$ reflects interesting properties of $G$, and at times completely determines it.
If $\Gamma$ is the $GK$-graph of a group of a given type then we say that $\Gamma$ is realizable as the $GK$-graph of groups of that type.
Since the prime spectrum of a solvable \cut group has at most $4$ primes, the $GK$- graphs realizable by solvable \cut groups have finite number of choices. In \cite{BKMdR23}, it is observed that the choices of $GK$-graphs realizable by solvable \cut groups are quite restricted. We present the results from \cite{BKMdR23}.

\begin{theorem}The $GK$-graph of any finite solvable \cut group is one of the graphs in Figure~1. Further, each $GK$-graph in Figure 1 is realizable by some solvable \cut group, except possibly for the ones appearing in the last row.


	\begin{figure}[h!]
		
		\begin{tabular}{|c|c|}
			\hline
			& Graphs \\\hline
			1 vertex & \emph{(a)}\hspace{-0.8cm}
			\begin{subfigure}{.15\textwidth}
				\centering 
				\begin{tikzpicture};
					\node[label=west:{$2$}] at (0,0) (2){};
					\foreach \p in {2}{
						\draw[fill=black] (\p) circle (0.075cm);
					}
				\end{tikzpicture}
			\end{subfigure}
			\hspace{3cm} 
			\emph{(b)}\hspace{-0.8cm}
			\begin{subfigure}{.15\textwidth}
				\centering 
				\begin{tikzpicture}
					\node[label=east:{$3$}] at (0.5,1) (3){};
					\foreach \p in {3}{
						\draw[fill=black] (\p) circle (0.075cm);
					}
				\end{tikzpicture}
			\end{subfigure} 
			\\\hline
			2 vertices & \emph{(c)}
			\hspace{-0.5cm}
			\begin{subfigure}{.15\textwidth}
				\centering 
				\begin{tikzpicture}
					\node[label=west:{$2$}] at (0,1) (2){};
					\node[label=east:{$3$}] at (0.5,1) (3){};
					\foreach \p in {2,3}{
						\draw[fill=black] (\p) circle (0.075cm);
					}
				\end{tikzpicture}
			\end{subfigure}
			\emph{(d)}
			\hspace{-0.5cm}
			\begin{subfigure}{.15\textwidth}
				\centering 
				\begin{tikzpicture}
					\node[label=west:{$2$}] at (0,1) (2){};
					\node[label=east:{$3$}] at (0.5,1) (3){};
					\foreach \p in {2,3}{
						\draw[fill=black] (\p) circle (0.075cm);
					}
					\draw (2)--(3);
				\end{tikzpicture}
			\end{subfigure} 
			\emph{(e)}
			\hspace{-1cm}
			\begin{subfigure}{.15\textwidth}
				\centering
				\begin{tikzpicture}
					\node[label=west:{$2$}] at (0,0.5) (2){};
					\node[label=west:{$5$}] at (0,0) (3){};
					\foreach \p in {2,3}{
						\draw[fill=black] (\p) circle (0.075cm);
					}
				\end{tikzpicture}
			\end{subfigure}
			\emph{(f)}
			\hspace{-1cm}
			\begin{subfigure}{.15\textwidth}
				\centering 
				\begin{tikzpicture}
					\node[label=west:{$2$}] at (0,0.5) (2){};
					\node[label=west:{$5$}] at (0,0) (5){};
					\foreach \p in {2,3}{
						\draw[fill=black] (\p) circle (0.075cm);
					}
					\draw (2)--(3);
				\end{tikzpicture}
			\end{subfigure}
			\emph{(g)}
			\hspace{-1cm}
			\begin{subfigure}{.15\textwidth}
				\centering 
				\begin{tikzpicture}
					\node[label=east:{$3$}] at (0.5,0.5) (3){};
					\node[label=east:{$7$}] at (0.5,0) (7){};
					\foreach \p in {3,7}{
						\draw[fill=black] (\p) circle (0.075cm);
					}
				\end{tikzpicture}
			\end{subfigure}
			\\\hline
			& 
			\emph{(h)}\hspace{-0.3cm}
			\begin{subfigure}{.15\textwidth}
				\centering
				\begin{tikzpicture}
					\node[label=west:{$2$}] at (0,0.5) (2){};
					\node[label=east:{$3$}] at (0.5,0.5) (3){};
					\node[label=west:{$5$}] at (0,0) (5){};
					\foreach \p in {2,3,5}{
						\draw[fill=black] (\p) circle (0.075cm);
					}
					\draw (2)--(3);
				\end{tikzpicture}
			\end{subfigure}
			\emph{(i)}\hspace{-0.3cm}
			\begin{subfigure}{.15\textwidth}
				\centering
				\begin{tikzpicture}
					\node[label=west:{$2$}] at (0,0.5) (2){};
					\node[label=east:{$3$}] at (0.5,0.5) (3){};
					\node[label=west:{$5$}] at (0,0) (5){};
					\foreach \p in {2,3,5}{
						\draw[fill=black] (\p) circle (0.075cm);
					}
					\draw (2)--(3);
					\draw (2)--(5);
				\end{tikzpicture}
			\end{subfigure} 
			\emph{(j)}\hspace{-0.3cm}
			\begin{subfigure}{.15\textwidth}
				\centering
				\begin{tikzpicture}
					\node[label=west:{$2$}] at (0,0.5) (2){};
					\node[label=east:{$3$}] at (0.5,0.5) (3){};
					\node[label=west:{$5$}] at (0,0) (5){};
					\foreach \p in {2,3,5}{
						\draw[fill=black] (\p) circle (0.075cm);
					}
					\draw (2)--(3);
					\draw (3)--(5);
				\end{tikzpicture}
			\end{subfigure}
			\emph{(k)}\hspace{-0.3cm}
			\begin{subfigure}{.15\textwidth}
				\centering
				\begin{tikzpicture}
					\node[label=west:{$2$}] at (0,0.5) (2){};
					\node[label=east:{$3$}] at (0.5,0.5) (3){};
					\node[label=west:{$5$}] at (0,0) (5){};
					\foreach \p in {2,3,5}{
						\draw[fill=black] (\p) circle (0.075cm);
					}
					\draw (2)--(3);
					\draw (2)--(5);
					\draw (3)--(5);
				\end{tikzpicture}
			\end{subfigure}
		\\			3 vertices & \\
			& \emph{(l)}\hspace{-0.3cm}
			\begin{subfigure}{.15\textwidth}
				\centering 
				\begin{tikzpicture}
					\node[label=west:{$2$}] at (0,0.5) (2){};
					\node[label=east:{$3$}] at (0.5,0.5) (3){};
					\node[label=east:{$7$}] at (0.5,0) (7){};
					\foreach \p in {2,3,7}{
						\draw[fill=black] (\p) circle (0.075cm);
					}
					\draw (2)--(3);
				\end{tikzpicture}
			\end{subfigure}
			\emph{(m)}\hspace{-0.3cm}
			\begin{subfigure}{.15\textwidth}
				\centering
				\begin{tikzpicture}
					\node[label=west:{$2$}] at (0,0.5) (2){};
					\node[label=east:{$3$}] at (0.5,0.5) (3){};
					\node[label=east:{$7$}] at (0.5,0) (7){};
					\foreach \p in {2,3,7}{
						\draw[fill=black] (\p) circle (0.075cm);
					}
					\draw (2)--(3);
					\draw (2)--(7);
				\end{tikzpicture}
			\end{subfigure}
			\emph{(n)}\hspace{-0.3cm}
			\begin{subfigure}{.15\textwidth}
				\centering
				\begin{tikzpicture}
					\node[label=west:{$2$}] at (0,0.5) (2){};
					\node[label=east:{$3$}] at (0.5,0.5) (3){};
					\node[label=east:{$7$}] at (0.5,0) (7){};
					\foreach \p in {2,3,7}{
						\draw[fill=black] (\p) circle (0.075cm);
					}
					\draw (2)--(3);
					\draw (3)--(7);
				\end{tikzpicture}
			\end{subfigure} 
			\emph{(o)}\hspace{-0.3cm}
			\begin{subfigure}{.15\textwidth}
				\centering 
				\begin{tikzpicture}
					\node[label=west:{$2$}] at (0,0.5) (2){};
					\node[label=east:{$3$}] at (0.5,0.5) (3){};
					\node[label=east:{$7$}] at (0.5,0) (7){};
					\foreach \p in {2,3,7}{
						\draw[fill=black] (\p) circle (0.075cm);
					}
					\draw (2)--(3);
					\draw (2)--(7);
					\draw (3)--(7);
				\end{tikzpicture}
			\end{subfigure}
			\\\hline

	4 vertices&  
		\begin{subfigure}{0\textwidth}
			\setcounter{subfigure}{15}
		\end{subfigure}
		\emph{(p)}\hspace{-0.4cm}
		\begin{subfigure}{.15\textwidth}
			\centering
			\begin{tikzpicture}
				\node[label=west:{$2$}] at (0,0.5) (2){};
				\node[label=east:{$3$}] at (0.5, 0.5) (3){};
				\node[label=west:{$5$}] at (0, 0) (5){};
				\node[label=east:{$7$}] at (0.5, 0) (7){};
				\foreach \p in {2,3,5,7}{
					\draw[fill=black] (\p) circle (0.075cm);
				}
				\draw (2)--(3);
				\draw (2)--(7);;
				\draw (3)--(5);
				\draw (5)--(7);
			\end{tikzpicture}
		\end{subfigure}
		 \emph{(q)}\hspace{-0.4cm}
		\begin{subfigure}{.15\textwidth}
			\centering
			\begin{tikzpicture}
				\node[label=west:{$2$}] at (0,0.5) (2){};
				\node[label=east:{$3$}] at (0.5, 0.5) (3){};
				\node[label=west:{$5$}] at (0, 0) (5){};
				\node[label=east:{$7$}] at (0.5, 0) (7){};
				\foreach \p in {2,3,5,7}{
					\draw[fill=black] (\p) circle (0.075cm);
				}
				\draw (2)--(3);
				\draw (2)--(5);
				\draw (2)--(7);;
				\draw (3)--(5);
				\draw (5)--(7);
			\end{tikzpicture}
		\end{subfigure}
		\emph{(r)}\hspace{-0.4cm}
		\begin{subfigure}{.15\textwidth}
			\centering
			\begin{tikzpicture}
				\node[label=west:{$2$}] at (0,0.5) (2){};
				\node[label=east:{$3$}] at (0.5, 0.5) (3){};
				\node[label=west:{$5$}] at (0, 0) (5){};
				\node[label=east:{$7$}] at (0.5, 0) (7){};
				\foreach \p in {2,3,5,7}{
					\draw[fill=black] (\p) circle (0.075cm);
				}
				\draw (2)--(3);
				\draw (2)--(5);
				\draw (2)--(7);;
				\draw (3)--(5);
				\draw (3)--(7);
				\draw (5)--(7);
			\end{tikzpicture}
		\end{subfigure}
		\\\hline
		4 vertices {(possible)} & \emph{(s)}\hspace{-0.4cm}
		\begin{subfigure}{.15\textwidth}
			\centering
			\begin{tikzpicture}
				\node[label=west:{$2$}] at (0,0.5) (2){};
				\node[label=east:{$3$}] at (0.5, 0.5) (3){};
				\node[label=west:{$5$}] at (0, 0) (5){};
				\node[label=east:{$7$}] at (0.5, 0) (7){};
				\foreach \p in {2,3,5,7}{
					\draw[fill=black] (\p) circle (.075cm);
				}
				\draw (2)--(3);
				\draw (2)--(7);
				\draw (3)--(5);
				\draw (3)--(7);
			\end{tikzpicture}
		\end{subfigure}
		\emph{(t)}\hspace{-0.4cm}
		\begin{subfigure}{.15\textwidth}
			\centering
			\begin{tikzpicture}
				\node[label=west:{$2$}] at (0,0.5) (2){};
				\node[label=east:{$3$}] at (0.5, 0.5) (3){};
				\node[label=west:{$5$}] at (0, 0) (5){};
				\node[label=east:{$7$}] at (0.5, 0) (7){};
				\foreach \p in {2,3,5,7}{
					\draw[fill=black] (\p) circle (.075cm);
				}
				\draw (2)--(3);
				\draw (2)--(5);
				\draw (2)--(7);
				\draw (3)--(5);
			\end{tikzpicture}
		\end{subfigure}
		\emph{(u)}\hspace{-0.4cm}
		\begin{subfigure}{.15\textwidth}
			\centering
			\begin{tikzpicture}
				\node[label=west:{$2$}] at (0,0.5) (2){};
				\node[label=east:{$3$}] at (0.5, 0.5) (3){};
				\node[label=west:{$5$}] at (0, 0) (5){};
				\node[label=east:{$7$}] at (0.5, 0) (7){};
				\foreach \p in {2,3,5,7}{
					\draw[fill=black] (\p) circle (.075cm);
				}
				\draw (2)--(3);
				\draw (2)--(7);
				\draw (3)--(5);
				\draw (3)--(7);
				\draw (5)--(7);
			\end{tikzpicture}
		\end{subfigure} 
		\emph{(v)}\hspace{-0.4cm}
		\begin{subfigure}{.15\textwidth}
			\centering
			\begin{tikzpicture}
				\node[label=west:{$2$}] at (0,0.5) (2){};
				\node[label=east:{$3$}] at (0.5, 0.5) (3){};
				\node[label=west:{$5$}] at (0, 0) (5){};
				\node[label=east:{$7$}] at (0.5, 0) (7){};
				\foreach \p in {2,3,5,7}{
					\draw[fill=black] (\p) circle (.075cm);
				}
				\draw (2)--(3);
				\draw (2)--(5);
				\draw (2)--(7);
				\draw (3)--(5);
				\draw (3)--(7);
			\end{tikzpicture}
		\end{subfigure}
	
		\\ \hline
	\end{tabular}
	\caption{\label{Four} GK-graphs of finite solvable \cut groups.}
\end{figure}

\end{theorem}

The answer for the following question is seeked. 
\begin{question}

Which of the graphs (s), (t), (u) or (v) in Figure 1 is/are realizable by a solvable \cut group? 
\end{question}
It may be noted that in \cite{BKMdR23}, the classification of $GK$-graphs realizable by varied classes of rational and \cut groups has been provided. These classes include abelian, nilpotent, metacyclic, metabelian, supersolvable, metanilpotent and Frobenius groups.
\begin{remark}
The $GK$-graphs of solvable rational groups have been completely classified recently (see \cite{BKMdR23} and \cite{DLdR24}). 
\end{remark}


\subsection{Classification of \cut groups:}\label{classification}In Section 2, we have seen several criteria to check if a given finite group is \cut or not. Given a particular class $C$, it is naturally seeked whether the properties of $C$ yield the complete list of \cut groups in $C$, or a refined criterion to answer this question. For instance, we observed in \Cref{abelian} that an abelian group $G$ is \cut if and only if its exponent divides $4$ or $6$. For $p$-groups, we observed in \Cref{PrimitiveRoots} that $p\in\{2,3\}$. Further a $2$-group is \cut if and only if every element $x\in G$ satisfies $x^3\sim x$ or $x^{-1}$ and a $3$-group is \cut if and only if $x^2\sim x^{-1}$ for every $x\in G$. For non-abelian groups of order $p^n$, $n< 5$, a complete list of $12$ \cut groups is computed in (\cite{MP18}, Proposition 2).  A nilpotent group $G$ is \cut if and only if it is either a \cut $p$-group or direct product of a rational $2$-group and a \cut $3$-group. Particularly, if $G$ is of nilpotency class at most $2$, then $G$ is \cut if and only if $\mathcal{Z}(G/N)$ is of exponent $1,2,3,4$ or $6$ for every quotient group $G/N$ of $G$ (\cite{BCJM21}, Theorem 4.1). If $G$ is an odd order group, then $G$ is \cut, if and only if, 
every element $x$ has order either $7$ or a power of $3$ and satisfies  $ x^5 \sim  x^{-1}$ (\cite{Mah18}, Theorem 1). If $G$ is an even order solvable group, then similar criterion exists when the group is known to have elements of prime power order only (\cite{Mah18}, Theorem 2). It shall be naturally of interest to extend the result for all solvable groups. Particularly, for metacyclic groups, we have a complete list (upto isomorphism, and with explicit presentations) of $46$ \cut groups (\cite{BMP17}, Theorem 5). Among non-solvable groups, a substantial class of \cut groups is that of Symmetric groups, which are all known to be rational. As also pointed in the introduction, the alternating \cut groups have been of interest since the very beginning of study of \cut groups. It is known that $A_n$ is \cut if and only if $n\in\{1,2,3,4,7,8,9,12\}$. The complete list of simple \cut groups is deducable from \Cref{CompositionFactors}, by a direct \texttt{GAP} check, and is also provided in \cite{BCJM21}. The Frobenius \cut groups are classified in \cite{Bac18}. In particular, all Camina \cut groups are classified, as Camina $p$-groups are already \cut.

\begin{remark}
	The classification of \cut groups clearly indicates the existence of \cut groups in varied classes. The class of \cut groups turns out to be immense.  In fact, in \texttt{GAP} check, about $86.62 \%$ of the groups of order at most $512$ and $78.55 \%$ of groups of order at most 1023 are found to be \cut groups, whereas these percentages are respectively $ 0.57 \%$ and $0.52\%$ for rational groups. Note that in (\cite{BCJM21}, Proposition 7.1), it is observed that $$ \lim_{n\rightarrow \infty} \frac{\ln c(p^n) }{\ln f(p^n)}=1,~\mathrm{for} ~p\in \{2,3\},$$ where $c(r)$ denotes the number of \cut groups of order $r$ and $f(r)$ denotes the number of all groups of order~$r$.
\end{remark}

\section{Generalisations of \cut groups: Semi-rational and quadratic rational groups} In this section, we study some classes of groups which were originally defined to generalise the notion of rational groups, but eventually are known to contain the class of \cut groups. For these classes of groups, we closely analyse the properties studied about \cut groups in previous sections, and point out relevant questions therein.\\ 

As introduced in 2.2, a group $G$ is said to be rational, if all its elements are rational, i.e., $\mathbb{Q}(x)=\mathbb{Q}$ for every $x\in G$, or equivalently, all generators of $\langle x\rangle$ lie in the same conjugacy class; equivalently, every irreducible character of $G$ is rational, i.e., $\mathbb{Q}(\chi)=\mathbb{Q}$ for every $\chi \in \Irr(G)$. The class of rational groups is well known and has been well studied. The properties of rational groups naturally motivate the questions on \cut groups. Furthermore, generalising the notion of rational groups, Chillag et al. \cite{CD10} introduced semi-rational groups, of which \cut groups form a subclass. In this context, \cut groups are usually referred to as inverse semi-rational groups. It is natural to seek which properties of \cut groups are retained by semi-rational groups, or such similar classes. 



Let $G$ be a group. As defined in 2.2, an element $x\in G$ is said to be semi-rational if all generators of $\langle x\rangle$ are contained in at most two conjugacy classes, say $C_x$ and $C_{x^{m_x}}$ ($m_x \in \mathbb{Z}$) and a group $G$ is said to be semi-rational, if every element of $G$ is semi-rational. In particular, if there exists an integer $m$ such that for every element $x\in G$, the generators of $\langle x\rangle$ are contained in $C_x \cup C_{x^{m}}$, then $G$ is said to be uniformly semi-rational group or $m$ semi-rational. For instance, a $\cut$ group $G$ is a uniformly semi-rational group with $m=-1$, or a $-1$ semi-rational group.

A group $G$ is said to be quadratic rational if for every $\chi\in \Irr(G)$, the character field $\mathbb{Q}(\chi)$ is either $\mathbb{Q}$ or a quadratic extension of $\mathbb{Q}$, i.e., each irreducible character of $G$ is rational or quadratic \cite{Ten12}. Indeed, a \cut group is a quadratic rational group.  

In \cite{Ten12}, it has been observed that if $G$ is a semi-rational group, then for every $x\in G$, $\mathbb{Q}(x)$ is either $\mathbb{Q}$ or its quadratic extension. As noted earlier, the field $\mathbb{Q}(x)$ is formed by adjoining the elements of a column (the column of conjugacy class of $x$) in the character table of $G$, whereas the field $\mathbb{Q}(\chi)$ is formed by adjoining the elements of a row of its character table. Hence, the classes of semi-rational and quadratic rational groups are not the same, to begin with. In fact, the group $G_1$ in \Cref{G1} is semi-rational but not quadratic rational whereas the group $G_2$ in \Cref{G2} is quadratic rational but not semi-rational. Call a group to be quadratic semi-rational, if it is both semi-rational as well as quadratic rational. Clearly, every \cut group is a quadratic semi-rational group. We shall soon observe that the property of being quadratic semi-rational retains most properties of \cut groups, which do not hold for the class of quadratic rational or semi-rational groups, in general.
For this subsection we use the term inverse semi-rational for \cut groups, and for brevity, we use the following notation for the classes of groups indicated:\\

\begin{tabular}{llll}

$\mathbb{Q}$ groups	&&& rational groups\\
$QR$ &&	& quadratic rational groups\\

$SR$&&	& semi-rational groups  \\
	
$ISR$	&&& inverse semi-rational groups \\
	
	$USR$&&& uniformly semi-rational groups\\
	
	$QSR$& && quadratic semi-rational groups  \\
	&&&\\
\end{tabular}

We begin by observing that though semi-rational groups need not be quadratic rational, a uniformly semi-rational group is always quadratic rational. 

\begin{proposition}\label{USR}
	$USR \subseteq QR$
\end{proposition}
\begin{proof} 
	Let $G$ be an $m$ semi-rational group and let $\chi\in \Irr(G)$. Consider $\sigma\in\Gal(\mathbb{Q}(\chi)/\mathbb{Q})$ and extend $\sigma$ to an automorphism $\sigma_j: \zeta \rightarrow {\zeta }^j$ of $\mathbb{Q}(\zeta)$, where $\zeta$ is a $|G|^\mathrm{th}$ root of unity and $\gcd(j,|G|)=1$. Then ${\chi}^{\sigma_j}(x)=\chi (x^j)=\chi (x)$ or $\chi (x^m)$, $x\in G$. Now, $x$ being a semi-rational element in $G$ satisfies $x^{m^2}\sim x$. This implies, $(\chi^{\sigma_m})^{\sigma_j}(x)=\chi (x^m)$ or $\chi (x^{m^2})=\chi (x) $ or $\chi (x^m)$. This implies $\chi + \chi^{\sigma_m}=\chi^{\sigma_j} + (\chi^{\sigma_m})^{\sigma_j}$. As irreducible characters are linearly independent, $\chi^{\sigma_j}=\chi$ or $\chi^{\sigma_m}$. Consequently, $[\mathbb{Q}(\chi) : \mathbb{Q}]\leq 2$, i.e., $G$ is quadratic rational. 
\end{proof}
 The above result also follows from \cite{dRV24} with a different treatment.\\
 
It may be noted that $ISR \subsetneq USR \subsetneq QSR$. This follows from the fact that the dihedral group of order $10$ is a $2$ semi-rational group but not inverse semi-rational and all alternating groups are known to be quadratic semi-rational, but not all of them are uniformly semi-rational (see \cite{Ver24}, Theorem A). We thus have
$$\mathbb{Q} \mathrm{~groups}	\subsetneq ISR \subsetneq USR \subsetneq QSR = SR \cap QR.$$

However, we shall observe that for certain types of groups, these classes are exactly same.
\begin{proposition}\label{abelian_same}If $G$ is an abelian group or a group of odd order, then the following statements are equivalent. 

	\begin{enumerate}
		\item[\rm{(i)}] $G$ is inverse semi-rational.
		\item[\rm{(ii)}]  $G$ is uniformly semi-rational.
		\item[\rm{(iii)}] $G$ is quadratic semi-rational
		\item[\rm{(iv)}] $G$ is quadratic rational.
		\item[\rm{(v)}] $G$ is semi-rational.
		
	\end{enumerate}
\end{proposition}
\begin{proof}
	Since for any group $G$, $\rm(i)\implies (ii)\implies (iii)\implies (iv)$ and $\rm(i)\implies (v)$. We need to prove $\rm(iv)\implies \rm(i)$ and $\rm(v)\implies (i)$.
	
	If $G$ is an abelian group, then for any divisor $d$ of $|G|$, there exists an element $x_d\in G$ of order $d$, such that $[\mathbb{Q}(x_d):\mathbb{Q}]=\phi(d)$, where $\phi$ is Euler totient function. Also, there exists an irreducible character $\chi_d$, such that $[\mathbb{Q}(\chi_d):\mathbb{Q}]=\phi(d)$. Hence, if $G$ is semi-rational or quadratic rational, we must have $\phi(d)\leq 2$, equivalently, the exponent of $G$ must be $1$, $2$, $3$, $4$ or $6$, i.e., $G$ must be inverse semi-rational.
	
	If $G$ is an odd order group, then (v)$\implies$ (i) and  (iv) $\implies$(i)  have been respectively proved in (\cite {CD10}, Remark 13) and (\cite{PV25}, Proposition 2.5).
\end{proof}
\subsection{Equivalences:}

We have detailed in Section 2, that the definition of a finite \cut group has various equivalent forms. Analogously, for the generalised classes under consideration here, we observe the equivalent conditions.  We begin by recalling some equivalent statements for finite semi-rational groups, which have been given in (\cite{CD10}, Lemma 5) and (\cite{Ten12}, Lemma 1).
\begin{proposition}\label{semi-rational}
		Let $G$ be a finite group. The following statements are equivalent.
	\begin{enumerate}
		\item[\rm{(i)}] $G$ is a semi-rational group.
		\item[\rm{(ii)}] For every $x\in G$, $[\mathbb{Q}(x) : \mathbb{Q}]\leq 2.$
		\item[\rm{(iii)}]  For every $x\in G$, $[\Aut(\langle x\rangle ):B_G(x)]\leq 2$.

	\end{enumerate}
\end{proposition}
The equivalent statements for \cut groups are much more, owing to the fact that they are both semi-rational as well as  quadratic rational. We next list the properties for groups which are both semi-rational and quadratic rational. 
\begin{proposition}
	Let $G$ be a finite  group. The following statements are equivalent.
	\begin{enumerate}
		\item[\rm{(i)}] $G$ is a quadratic semi-rational group.
		\item[\rm{(ii)}] For each $x\in G$, there exists $\chi \in Irr(G)$ such that  $\mathbb{Q}(x)=\mathbb{Q}(\chi)$,  and  $[\mathbb{Q}(x) : \mathbb{Q}]\leq 2.$
		\item[\rm{(iii)}] For each $\chi\in \Irr(G)$, there exists $x \in G$ such that  $\mathbb{Q}(x)=\mathbb{Q}(\chi)$,  and  $[\mathbb{Q}(\chi) : \mathbb{Q}]\leq 2.$
		\item[\rm{(iv)}] For every  $x \in G$,  $[\Aut(\langle x\rangle ):B_G(x)]\leq 2$ and $\mathbb{Q}(x)=\mathbb{Q}(\chi)$ for some $\chi \in Irr(G)$.
		
	\end{enumerate}
\end{proposition}

\begin{proof}
 Firstly, (i)$\implies$ (ii). 
 Let $x\in G$. If $\mathbb{Q}(x)\neq \mathbb{Q}$, then there exists $\chi\in \Irr(G)$ such that $\mathbb{Q}(\chi(x))\neq \mathbb{Q}$, but $\mathbb{Q}(\chi(x))\subseteq \mathbb{Q}(x)$ and $[\mathbb{Q}(x): \mathbb{Q}]\leq 2$.	Therefore, we must have $\mathbb{Q}(\chi(x))= \mathbb{Q}(x)$. Now, $\mathbb{Q}(\chi(x))\subseteq \mathbb{Q}(\chi)$, and again as  $[\mathbb{Q}(\chi): \mathbb{Q}]\leq 2$, we obtain that $\mathbb{Q}(\chi)= \mathbb{Q}(x)$. With similar arguments and above proposition, all the equivalent statements hold good.
\end{proof}

Uniformly semi-rational or $m$ semi-rational groups are quadratic semi-rational. In terms of group actions, we can also state the following:
\begin{proposition}\label{m_semi_rational}
	Let $G$ be a finite group. The following statements are equivalent:
	\begin{enumerate}
		\item[\rm{(i)}] $G$ is an $m$-semi-rational group.
	\item [\rm{(ii)}]	For each $x \in G$, $\Aut(\langle x \rangle) = B_G(x)\langle \tau \rangle$ where $\tau \in \Aut(\langle x \rangle)$ is the map $x\rightarrow x^m$.
		\item[\rm{(iii)}] For every $\chi \in \Irr(G)$ and for every  $\sigma \in \Gal(\mathbb{Q}(\chi)/\mathbb{Q})$, $\chi^{\sigma} = \chi$ or  $\chi^{\sigma_m}$, where \linebreak$\chi^{\sigma_m}(x)= \chi(x^m)$.
	 	\end{enumerate}
\end{proposition}
%
%
%
%
\subsection{Subgroups and quotient groups:} 
We now study these classes for closures with respect to subgroups, quotient groups etc. Clearly, none of these can be expected to be subgroup closed, owing to the fact that all finite groups are subgroups of symmetric groups, which are known to be rational. Like \cut groups, we shall study the closure of these classes for the centre, Sylow subgroups, quotient groups and direct products.

\begin{proposition}\label{centre_quotient_closed_all}
	Let $\mathfrak{X}\in \{\mathbb{Q}~groups, ISR,USR,QSR,SR,QR\}$. If $G\in \mathfrak{X}$, then so does $\mathcal{Z}(G)$ and any homomorphic image of $G$.
\end{proposition} 

\begin{proof}
	Suppose first $\mathfrak{X}=\mathbb{Q} ~$groups. If $G\in \mathfrak{X}$, i.e., $G$ is a rational group, then any $x\in\mathcal{Z}(G)$, satisfies $x^j\sim x$, for all $j$ such that $\gcd(j,|x|)=1$. Since $x$ is a central element, this condition implies that $|x|\leq 2$. In particular,   $\mathcal{Z}(G)$ is an elementary abelian $2$-group, which is clearly a rational group.
	
	Likewise, if $\mathfrak{X}\in \{ISR,USR,QSR,SR\}$ and $G\in \mathfrak{X}$, then for any $x\in\mathcal{Z}(G)$, $x$ satisfies that for any $j$ with $\gcd(j,|x|)=1$, either $x^j\sim x$ or $x^j\sim x^{m_x}$, for some $m_x\in \mathbb{Z}$. This yields $|x|\in \{1,2,3,4,6\}$. Hence, $\mathcal{Z}(G)\in \mathfrak{X}$, in view of Proposition \ref{abelian_same}.
	
	Now, if $\mathfrak{X}=QR$ and $G\in \mathfrak{X}$, i.e, $G$ is a quadratic rational group, then for $x\in \mathcal{Z}(G)$, for any $\chi\in \Irr(G)$, we have that $\chi(x)=\zeta$, where $\zeta$ is $|x|^\mathrm{th}$ root of unity. Since, $\chi(x)=\zeta\in \mathbb{Q}(\chi)$ and $G$ being quadratic, $[\mathbb{Q}(\chi):\mathbb{Q}]\leq 2$. Therefore, we must have $|x|\in\{1,2,3,4,6\}$ and hence again by Proposition \ref{abelian_same}, $\mathcal{Z}(G)$ must be quadratic rational group.  
	
	If $N\unlhd G$, then every irreducible character of $G/N$ may be viewed as an irreducible character of $G$. If $G\in \mathfrak{X}$, where $\mathfrak{X}\in \{\mathbb{Q},QR, ISR\}$, then the condition on possible character fields of $G$ coming from the fact that $G\in \mathfrak{X}$ ensures that $G/N\in \mathfrak{X}$.
	
	Now, if $\mathfrak{X}\in \{SR,QSR,USR\}$, then for any $\overline{x}=xN\in G/N$, consider $j$ such that $\gcd(j,|x|)=1$. Then there exist $g\in G$ such that $g^{-1}x^jg=x$ or $x^m$, where $m\in \mathbb{Z}$. Hence, $\overline{g}^{-1}\overline{x^j}\,\overline{g}=\overline{x}$ or $\overline{x^m}$, i.e., $\overline{x}^j\sim \overline{x}$ or $\overline{x}^m$ in $G/N$. Therefor, $G/N\in \mathfrak{X}$, if $G\in \mathfrak{X}$. 
\end{proof}

The quotient closed property for these classes along with \Cref{abelian_same} and (\cite{BCJM21}, Theorem 4.1) yields next result, which extends  \Cref{abelian_same} by including nilpotent groups of class $2$. 

\begin{theorem}\label{same_nilpotent2_odd}If $G$ is a nilpotent group of class at most $2$, or a group of odd order, then the following statements are equivalent:
	\begin{enumerate}
	\item[\rm{(i)}] $G$ is inverse semi-rational.
	\item[\rm{(ii)}]  $G$ is uniformly semi-rational.
	\item[\rm{(iii)}] $G$ is quadratic semi-rational
	\item[\rm{(iv)}] $G$ is quadratic rational.
	\item[\rm{(v)}] $G$ is semi-rational.
	
\end{enumerate}
\begin{remark}It is difficult to say if above theorem may be further extended to bigger classes.\begin{enumerate}
		\item[\rm{(i)}] \Cref{same_nilpotent2_odd} holds no longer for nilpotency class $3$. For instance, groups $G_1$ and $G_2$ in Examples \ref{G1} and \ref{G2} serve as counter examples.
		\item[\rm{(ii)}] As pointed in (\cite{BCJM21}, Remark 4.2), \texttt{SmallGroup(81,8)} fails to be inverse semi-rational though it satisfies the hypothesis of (\cite{BCJM21}, Theorem 4.1). This group even fails to be semi-rational or quadratic rational.    
	\end{enumerate}

\end{remark}
\end{theorem}
 Before proceeding further, we first study these classes with respect to the direct products.
\subsection{Direct products:}

It is easy to observe that the class of $\mathbb{Q}$ groups is closed with respect to direct products. If $G\in ISR$, we have \Cref{cut_direct}. We now see that the situation is similar for other classes as well.
\begin{proposition} 	Let $\mathfrak{X}\in \{USR,QSR,SR,QR\}$ and let $G_1,G_2\in \mathfrak{X}$. Then
	the direct product $G_1\times G_2\in \mathfrak{X}$,  if and only if, one of the following holds:
	\begin{enumerate}
		\item[\rm{(i)}] Either $G_1$ or $G_2$ is rational.
		\item[\rm{(ii)}] $\mathbb{Q}(G_1)=\mathbb{Q}(G_2)=\mathbb{Q}(\sqrt{d})$, for some $d\in\mathbb{Z}*$, where $\mathbb{Z}*$ denotes the set of square free integers.
	\end{enumerate}
	
\end{proposition}
\begin{proof} The statement holds for $\mathfrak{X}=QSR$, if it holds for both $\mathfrak{X}=QR$ as well as $\mathfrak{X}=SR$. Thus, we prove the result for these two cases. 
	
	Let $\mathfrak{X}=QR$ and let $G_1,G_2 \in \mathfrak{X}$ be such that $G_1\times G_2\in \mathfrak{X}$. Now, if (i) does not hold, i.e., none of $G_1$ or $G_2$ is rational, then there exist $\chi_1\in \Irr(G_1)$ and $\chi_2\in \Irr(G_2)$ such that $\mathbb{Q}(\chi_1)=\mathbb{Q}(\sqrt{d_1})$ and $\mathbb{Q}(\chi_2)=\mathbb{Q}(\sqrt{d_2})$ for some $d_1,d_2\in \mathbb{Z}*$. Since $\chi_1\times \chi_2\in \Irr(G_1\times G_2)$, we have that $\mathbb{Q}(\sqrt{d_1},\sqrt{d_2})\subseteq \mathbb{Q}(\chi_1\times \chi_2)$, but $G_1\times G_2$ being a quadratic rational group implies that $\mathbb{Q}(\chi_1\times \chi_2)$ must be $\mathbb{Q}(\sqrt{d})$ for some $d\in \mathbb{Z}_*.$ Therefore, we must have $d_1=d_2=d$. Further, if $\chi_1'$ is any character of $G_1$ such that $\mathbb{Q}(\chi_1')=\mathbb{Q}(\sqrt{{d_1}'})$, ${d_1}'\in \mathbb{Z}_*$, then ${d_1}'=d.$ Thus, $\mathbb{Q}(G_1)=\mathbb{Q}(\sqrt{d})$, and similarly $\mathbb{Q}(G_2)=\mathbb{Q}(\sqrt{d})$. Consequently, $\mathbb{Q}(G_1)=\mathbb{Q}(G_2)=\mathbb{Q}(\sqrt{d})=\mathbb{Q}(G_1 \times G_2).$ 
	
	Next, if $\mathfrak{X}=SR$  and $G_1,G_2 \in \mathfrak{X}$. Let $G_1\times G_2\in \mathfrak{X}$. If (i) does not hold, then there exist $x_1\in G_1$ and $x_2\in G_2$ such that $\mathbb{Q}(x_1)=\mathbb{Q}(\sqrt{d_1})$ and $\mathbb{Q}(x_2)=\mathbb{Q}(\sqrt{d_2})$ for some $d_1,d_2\in \mathbb{Z}*$. As  $\mathbb{Q}(\sqrt{d_1},\sqrt{d_2})\subseteq \mathbb{Q}((x_1,x_2))$,  which implies that $\mathbb{Q}((x_1,x_2))$ must be of the form $\mathbb{Q}(\sqrt{d})$ for some $d\in \mathbb{Z}*$, with $d=d_1=d_2$, $G_1 \times G_2$ being semi-rational group. Consequently, $\mathbb{Q}(G_1)=\mathbb{Q}(G_2)=\mathbb{Q}(\sqrt{d})=\mathbb{Q}(G_1 \times G_2),$ i.e., (ii) holds. 
	
		It is easy to see that if (i) or (ii) holds, then $G_1 \times G_2\in \mathfrak{X}$, in both cases when $\mathfrak{X}=SR$ or $QR$.
		
	Now, let $\mathfrak{X}=USR$  and $G_1, G_2 \in \mathfrak{X}$. If $G_1\times G_2\in \mathfrak{X}$, then $G_1\times G_2$ being semi-rational, either (i) or (ii) holds. For the converse, if (i) holds, $G_1\times G_2$ is clearly uniformly semi-rational, and if (ii) holds, then $\mathbb{Q}(G_1\times G_2)=\mathbb{Q}(\sqrt{d})$, so that every irreducible character of  $G_1\times G_2$ has character field $\mathbb{Q}$ or $\mathbb{Q}(\sqrt{d})$. Hence, by \Cref{m_semi_rational}, $G_1\times G_2 \in \mathfrak{X}$. 
	
\end{proof}

Note that the case $\mathfrak{X}=USR$, of above theorem is proved in (\cite{dRV24}, Proposition 3.15), but we have provided a generic proof.\\
 
The following now follows inductively.
\begin{corollary} Let $\mathfrak{X}\in\{USR,QR,SR,QSR\}$ and let $G_1,G_2,...,G_n\in \mathfrak{X}$. Then, the direct product $G_1\times G_2\times \ldots \times G_n \in \mathfrak{X}$, if and only if, for all non-rational groups $G_i$, $i\in I \subseteq \{1,2,...,n\}$, we have that
	$\mathbb{Q}(G_i)=\mathbb{Q}(\sqrt{d})$, for a fixed square free integer $d$. In particular, if $G\in \mathfrak{X}$, then so does $G\times C_2^n$, $n\in \mathbb{N}$.
\end{corollary}

\subsection{Rational conjugacy classes:} We have observed in 3.1.1 that for an inverse semi-rational group $G$, the number of rational irreducible characters of $G$ is equal to the number of rational conjugacy classes of $G$, and this does not hold good any more if $G$ is assumed to be semi-rational or quadratic rational. However, we observe that result holds good if $G$ is known to be both semi-rational as well as quadratic rational. This follows exactly from the proof of (\cite{BCJM21}, Theorem 3.1), except only by replacing the character value fields $\mathbb{Q}(\sqrt{-d})$, where $d\in \mathbb{N}$ by  $\mathbb{Q}(\sqrt{d})$, 
where $d\in \mathbb{Z}*$.

\begin{theorem}
	Let $G\in QSR$, then the number of rational irreducible characters of $G$ is equal to the number of rational conjugacy classes of $G$.
\end{theorem}

The following corollaries are immediate from above theorem.
\begin{corollary}
	If $G$ is an uniformly semi-rational group, then the number of rational irreducible characters of $G$ is equal to the number of rational conjugacy classes of $G$.
\end{corollary}
\begin{corollary}
	Let $G\in QSR$, then the number of quadratic irreducible characters of $G$ is equal to the number of quadratic conjugacy classes of $G$.
\end{corollary}
\subsection{Prime spectrum:}

 We saw in the last section that for a finite group $G$, being \cut has strong restrictions on the prime spectrum $\pi(G)$. We study the restrictions on $\pi(G)$  when $G$ is in one of the other classes in consideration and also consider Sylow $p$-subgroups of $G$, for $p\in \pi(G)$.

Recall that if $G$ is a rational group then $2\in\pi{(G)}$, and if $G$ is an inverse semi-rational group, then either $2\in \pi(G)$ or $3\in\pi(G)$. Generalising \Cref{prime_spectrum_cut}, we have the following result.

\begin{theorem}\label{prime_spectrum_all}For $G\in \mathfrak{X}$, where $\mathfrak{X}\in\{ISR,USR,QR,SR,QSR\}$, we have
	
	\begin{enumerate}
		\item[\rm{(i)}] Either $2\in \pi(G)$ or $3\in\pi(G)$.
		\item[\rm{(ii)}] If $G$ is nilpotent, then $\pi(G)\subseteq \{2,3\}$.
		\item[\rm{(iii)}] If $G$ is an odd order group, then $\pi(G)\subseteq \{3,7\}$.
		\item[\rm{(iv)}] If $G$ is solvable, then 
		\begin{enumerate}
			\item[\rm{(a)}] $\pi(G)\subseteq \{2,3,5,7\}$, if $\mathfrak{X}=ISR$.
			\item[\rm{(b)}] $\pi(G)\subseteq \{2,3,5,7,13\}$, if $\mathfrak{X}\in \{USR,QR,QSR\}$.
			\item[\rm{(c)}] $\pi(G)\subseteq \{2,3,5,7,13,17\}$, if $\mathfrak{X}=SR$.
		\end{enumerate}
	\end{enumerate} 
	
\end{theorem}
\begin{proof}The statements (i)-(ii) follow from the arguments in proof of (\cite{BMP17}, Theorem 1), keeping in view \Cref{centre_quotient_closed_all} and \Cref{same_nilpotent2_odd}. Additionally, (iii) as well as (iv)(a) follows from \Cref{same_nilpotent2_odd} and \Cref{prime_spectrum_cut}, whereas (iv)(b) follows from (\cite{Ten12}, Theorem A). Also, (iv)(c) is proved in (\cite{CD10}, Theorem 2).  
	\end{proof}
 \begin{remark}
 	The prime spectrums given in above theorem are optimal in the sense that none of the primes can be dropped for respective classes, except possibly $17$, for semi-rational groups. For instance, the existence of a group of order divisible by any prime in $\{2,3,5,7,13\}$ in the required class, follows easily from Main theorem of \cite{PV25} and (\cite{BMP17}, Theorem 5).  
 \end{remark}

\begin{question}Does there exist a solvable semi-rational group $G$ such that $17$ divides $|G|$?
	\end{question}
Analogous to the embedding questions for solvable subgroups of suitable prime spectrums, in \cut groups, one may ask relevant questions for the classes discussed here. 
\begin{question}
	Can every $13$-group be embedded in a solvable group of $\mathfrak{X}$, where \linebreak $\mathfrak{X}\in \{USR,SR, QSR,QR\}$?
\end{question} 

As a follow up to above question, one may look into various versions of questions, on the line of \Cref{embedding_cut}, aiming at the following question.
\begin{question}
	For what $\mathfrak{X} \in \{USR,SR, QSR,QR\}$, a solvable $\{2,3,5,7,13\}$-group can be embedded in a solvable group from $\mathfrak{X}$?
\end{question}
In view of the possible prime spectrum for a nilpotent group in a given class, Theorem 3 of \cite{Mah18} may be generalised as following:
\begin{theorem}Let $\mathfrak{X} \in \{\mathrm{ISR},\mathrm{USR}, \mathrm{SR}, \mathrm{QSR}\}. $ A finite nilpotent group $G$ is in $\mathfrak{X}$,  if and only if, the following conditions hold:
	\begin{enumerate}
		\item[\rm{(i)}] $G$ is a $2$-group in $\mathfrak{X}$.
		\item[\rm{(ii)}] $G$ is a $3$-group in $\mathfrak{X}$. Equivalently, for all $x \in G$, $x$ satisfies $x^2 \sim x^{-1}$.
		\item[\rm{(iii)}] $G \cong H \times K$, where $H$ is a rational $2$-group and $K$ satisfies condition \rm{(ii)}.
	\end{enumerate}
\end{theorem}
\subsection{Sylow subgroups:}
Though being  inverse semi-rational is not a subgroup closed property, we have seen that the situation is not very clear for Sylow subgroups of \cut groups. We now observe similar results for semi-rational and in particular for uniformly semi-rational groups. Analogous to \Cref{Sylow p-subgroup}, we have the following:

	\begin{theorem}\label{Sylow p-subgroup_SR}
	Let $G$ be a semi-rational (an $m$ semi-rational) group and let $P $ be a Sylow $p$-subgroup of $G$, where $p\in \{2,3\}$. Then $P$ is also a semi-rational (an $m$ semi-rational), if:
	\begin{enumerate}
		\item[\rm{(i)}] $P$ is an abelian or a normal subgroup of $G$.
		\item[\rm{(ii)}] $G$ is a supersolvable group.
		\item[\rm{(iii)}] $G$ is a Frobenius group. 
		\item[\rm{(iv)}] $G$ is a simple group, except possibly when $G$ is either $M$ (Monster group) and $p=2$ or $G=B$({Baby} Monster group) and $p\in\{2,3\}$. 
		\item[\rm{(v)}] $G$ is a group of odd order.
		\item[\rm{(vi)}] $G$ is a symmetric group.
	\end{enumerate}

\end{theorem}

Note that Sylow $2$-subgroups of a \cut group need not be \cut again. The counter examples given in 3.1.1, namely \texttt{SmallGroup[384,18033]} and \texttt{SmallGroup[384,18040]} do not serve the purpose here anymore, as their Sylow $2$-subgroups turn out to be semi-rational, rather uniformly semi-rational. Hence, it shall be interesting to investigate the following question:
\begin{question}
	Is a Sylow $p$-subgroup ($p\in\{2,3\}$) of a semi-rational (uniformly semi-rational) group again a semi-rational (uniformly semi-rational) group?
\end{question}

\subsection{Classification of groups:} As pointed in \Cref{classification}, a complete classification of inverse semi-rational groups is known for various substantial classes of finite groups. These classes include abelian groups, metacyclic groups, Frobenius groups, odd order groups, simple groups and certain $p$-groups. In this section, for the same classes, we would like to classify the groups in $\mathfrak{X}$, where $\mathfrak{X}\in \{SR,QR,QSR,USR\}$. To begin with, we recall that in view of \Cref{same_nilpotent2_odd}, these classes are all same for abelian groups, rather for nilpotent groups of class at most $2$ and for odd order groups. For $2$-groups, however, the classes are all distinct. For instance, cyclic group of order $4$ is an inverse semi-rational group which is not rational; quaternion group of order $16$ (\texttt{SmallGroup[16,9]}) is a $3$ semi-rational but not inverse semi-rational; \texttt{SmallGroup[128,417]} is a quadratic semi-rational group which is not uniformly semi-rational; group $G_1$ in \Cref{G1}  is semi-rational but not quadratic rational and vice versa for $G_2$ in \Cref{G2}.\\

 We now consider other classes such as metacyclic, Frobenius, simple etc. 

\subsubsection{Metacyclic groups:}

In \Cref{metacyclic_SR}, we give a complete classification of finite non-abelian metacyclic groups which are semi-rational, thereby generalising (\cite{BMP17}, Theorem 5). As inverse semi-rational groups are already listed in (\cite{BMP17}, Theorem 5), we enlist only the remaining ones, that is semi-rational groups which are not inverse semi-rational groups. The value of $m$ indicated along group presentation signifies that the group is $m$-semi-rational.
\begin{theorem}\label{metacyclic_SR}
	Let $G$ be a finite non-abelian metacyclic group given by
	$$ G= \langle a, b ~|~ a^n=1, b^t= a^l, a^b=a^r\rangle$$ where $n, t, r, l$ are natural numbers such that $r^t \equiv 1 (\mathrm{mod}~n), lr \equiv l (\mathrm{mod}~n)$ and $l$ divides $n$. Then, $G$ is a semi-rational group, if and only if, either $G$ is inverse semi-rational or $G$ is isomorphic to one of the following groups. Further, each semi-rational group, except two of them, is uniformly semi-rational ($m$ semi-rational).

%
%
%
%
	
\begin{quote}

	$\langle a, b \,| \, a^5=1, b^2= a^5, a^b=a^4\rangle, m=2$ \\
	$\langle a, b \,| \, a^8=1, b^2= a^4, a^b=a^7\rangle, m=3$ \\
	$\langle a, b \,| \, a^8=1, b^2= e, a^b=a^7\rangle, m=3$ \\ 
	$\langle a, b \,| \, a^{10}=1, b^2= a^5, a^b=a^9\rangle, m=3$ \\ 
	$\langle a, b \,| \, a^{10}=1, b^2= e, a^b=a^9\rangle, m=3$ \\ 
	$\langle a, b \,| \, a^{12}=1, b^2= a^6, a^b=a^{11}\rangle, m=7$ \\
	$\langle a, b \,| \, a^{12}=1, b^2= e, a^b=a^{11}\rangle, m=5$ \\ 
	$\langle a, b \,| \, a^8=1, b^4= e, a^b=a^7\rangle, m=3$ \\
	$\langle a, b \,| \, a^{10}=1, b^4= e, a^b=a^9\rangle, m=3$ \\
	$\langle a, b \,| \, a^{12}=1, b^4= e, a^b=a^{11}\rangle, m=7$ \\ 
	$\langle a, b \,| \, a^{13}=1, b^6= e, a^b=a^4\rangle, m=5$ \\  
 $\langle a, b \,| \, a^{21}=1, b^6= e, a^b=a^5\rangle, m=11$\\
	$\langle a, b \,| \, a^{26}=1, b^6= e, a^b=a^{17}\rangle, m=11$ \\ 
	$\langle a, b \,| \, a^{28}=1, b^6= a^{14}, a^b=a^3\rangle, m=5$ \\ 
	$\langle a, b \,| \, a^{28}=1, b^6= e, a^b=a^3\rangle, m=5$ \\
		$\langle a, b \,| \, a^{42}=1, b^6= e, a^b=a^5\rangle, m=11$ \\
		 $\langle a, b \,| \, a^8=1, b^4= a^4, a^b=a^3\rangle$\\
		 $\langle a, b \,| \, a^{12}=1, b^4= a^6, a^b=a^{11}\rangle$
\end{quote}
\end{theorem}
\begin{proof}
	Let $G$ be a metacyclic group with the given presentation. If $G$ is semi-rational, then so is $G/\langle a\rangle\cong \langle \overline{b}\rangle$. By \Cref{abelian_same}, \( t \in \{2, 3, 4, 6\} \). The
	number of $G$-conjugates of $a$ is  equal to $t$ and consequently we obtain
	that $|B_G(a)| = t$. Using \Cref{semi-rational} , we obtain that $\frac{\phi(n)}{t} \in\{1,2\}$ and thus $\phi(n) \in \{2,4,6,8,12\}$. These restrictions on parameters yield a finite number of metacyclic groups. 
%
	 Using \texttt{GAP} \cite{GAP4}, with a systematic check done on above finite restricted parameters, the result is obtained. 
\end{proof}
We have same classification of quadratic metacyclic groups, as discussed below.
\begin{theorem}\label{metacyclic_QR}
	A finite metacyclic group is quadratic rational if and only if it is uniformly semi-rational.
\end{theorem}
\begin{proof}
Suppose $G$ is non-abelian metacyclic group, with presentation in \Cref{metacyclic_SR}. If $G$ is quadratic rational, then $G/\langle a\rangle$ is also quadratic rational cyclic group of order $t $, so that $t \in  \{2, 3, 4, 6\}$. Observe that if $o$ denotes the multiplicative order of $r$ modulo $n$, then $(\langle a, b^{o} \rangle, \langle  b^o\rangle )$ is a strong Shoda pair of $G$ and the corresponding simple component is $M_o(\mathbb{F}),$ with $[\mathbb{F} : \mathbb{Q}]= \frac{\phi(n)}{o},$ where $\phi$ is the Euler totient function (\cite{OdRS04}, Proposition 3.4). As $G$ is a quadratic rational group,  we have that $\frac{\phi(n)}{o}\leq 2,$ i.e., $\frac{\phi(n)}{o} \in \{1, 2\}$.  Since $o|t$ and $t \in \{2, 3, 4, 6\} $, we have $\phi(n) \in \{2,4,6,8,12\}$. Consequently, we have same choices of parameters, as in the last theorem. Checking again for being quadratic rational, via \texttt{GAP}, we get the stated result.
\end{proof}
Now, we immediately have the following corollary.
\begin{corollary}\label{metacycilc_eq}
The following are equivalent statements for a finite metacyclic group $G$.
\begin{enumerate}
	
	\item[\rm{(i)}] $G$ is uniformly semi-rational.
\item[\rm{(ii)}]	$G$ is  quadratic rational.
\item[\rm{(iii)}]	$G$ is quadratic semi-rational.
\end{enumerate}
\end{corollary}
\subsubsection{Simple groups and Frobenius groups:}
Clearly, the alternating groups $A_n$ are both semi-rational and quadratic rational for every $n$. However, not every alternating group is uniformly semi-rational group. It is known that $A_n$ is uniformly semi-rational, if and only if $n \leq 22$ and $n \notin \{16, 21\}$  \cite{Ver24}. In contrast, for non-alternating finite simple groups, the notions of being semi-rational, quadratic rational, and uniformly semi-rational coincide. 

\begin{proposition}\label{Simple_QR_QSR}(\cite{Ver24}, Remark 3.1, \cite{PV25}, Remark 5.5)
If $G$ is a non-alternating simple group or a Frobenius group, then the following statements are equivalent :
\begin{enumerate}
\item[\rm{(i)}]  $G$ is uniformly semi-rational.
\item[\rm{(ii)}]$G$ is quadratic semi-rational
\item[\rm{(iii)}] $G$ is quadratic rational.
\item[\rm{(iv)}] $G$ is semi-rational.
 \end{enumerate}
\end{proposition}

Since Camina group is either a Frobenius group or a $p$-group \cite{DS96}, \Cref{Simple_QR_QSR} extends to the class of Camina groups and we have the following:

\begin{corollary} If $G$ is a Camina group, then the following statements are equivalent :
\begin{enumerate}
	\item[\rm{(i)}]  $G$ is uniformly semi-rational.
	\item[\rm{(ii)}] $G$ is quadratic semi-rational
	\item[\rm{(iii)}] $G$ is quadratic rational.
	\item[\rm{(iv)}] $G$ is semi-rational.
\end{enumerate}

\end{corollary} 
\subsection{GK-graphs:}
	The $GK$-graphs realizable by solvable inverse semi-rational groups have been studied. A key result used to classify $GK$-graphs is clearly the prime spectrum of $G$. In view of \Cref{prime_spectrum_all}, the following question(s) naturally makes(make) sense.	\begin{question} Given $\mathfrak{X} \in \{USR, QSR, QR,SR\}$, which all $GK$-graphs are realizable by finite solvable groups in $\mathfrak{X}$?
	\end{question}

	\subsection{Composition series:} \Cref{CompositionFactors} motivates the following question(s):
	
	\begin{question}Given $\mathfrak{X} \in \{USR, QSR,SR\}$, which all
		non-abelian simple groups can occur as composition factors of $G$, if  $G\in \mathfrak{X}$?
\end{question}
Note that the possible composition factors of $G$, when $G$ is a quadratic rational group have been provided in \cite{Tre17}.

\subsection{Degrees of $\mathbb{Q}(G)$:}

Chillag and Dolfi \cite{CD10} asked whether \( [\mathbb{Q}(G) : \mathbb{Q}] \) is bounded for solvable  semi-rational groups. Tent~\cite[Theorem~7]{Ten12} answered this affirmatively for all semi-rational groups.  In fact, from their proof the bound can be refined as follows:
\begin{theorem}[Tent~\cite{Ten12}]
	Let $G$ be a quadratic rational or semi-rational group. Assume that there are exactly $l$ distinct prime divisors of $ |G|$. Then
	$$
[\mathbb{Q}(G) : \mathbb{Q}] \leq 2^{l+1}.
	$$
\end{theorem}

This gives a bound in terms of the prime divisors of $|G|$ but not a uniform bound across all such groups. Alternating groups, which are both semi-rational and quadratic rational, demonstrate that we can not obtain a uniform bound for  $[\mathbb{Q}(G):\mathbb{Q}]$ in these classess.
Bachle~\cite{Bac19} asked whether there exists a natural class $\mathfrak{C}$ of groups containing all rational groups such that $ [\mathbb{Q}(G):\mathbb{Q}] \leq c $ for all  $G \in \mathfrak{C}$, where $c$ is a fixed constant. As stated in \Cref{degree_question}, this question remains open for \cut groups. 
 Since not all alternating groups are uniformly semi-rational, this leads naturally to the following question:
\begin{question}
	Is the degree of the extension $ \mathbb{Q}(G)$ over $\mathbb{Q} $ uniformly bounded for the class of uniformly semi-rational groups?
\end{question}
\subsection{Normalizer property:} We have observed that the \cut property implies the normalizer property. However, certain properties of more general classes may also be sufficient for the normalizer property to hold.

	\begin{question}
For what $\mathfrak{X}\in \{USR,QSR,SR,QR\}$, we have that $G\in\mathfrak{X} $ implies that $G$ has the normalizer property?
	\end{question}

	\subsection{The rank $\rho(G)$:} Recall that the torsion free group of $\mathcal{Z}(\U(\mathbb{Z}G))$ is a free abelian group of finite rank, which we denote by $\rho(G)$. Also, a group $G$ is \cut, i.e., inverse semi-rational if and only if $\rho(G)=0$, so that groups with rank $\rho(G)=0$ essentially belong to the classes considered in this section. As all alternating groups are quadratic rational as well as semi-rational and ranks of these groups can be arbitrarily large (see \cite{AKS08}), there is no bound on ranks for groups which are semi-rational or quadratic rational or quadratic semi-rational. However, alternating group $A_n$ is uniformly semi-rational only for finite choices of $n$, namely $n\leq 22$, $n\neq 16,21$ and hence the situation is different from previous cases. Clearly, USR contains groups of non-zero rank, but following question may be addressed.
	
	\begin{question}
	Is there a bound on free ranks of group of central units of integral group rings of uniformly semi-rational groups?
	\end{question}

\section*{Closing remarks}We have already observed the interesting properties enclosed in \cut groups and the potential that this topic holds because of the rich interplay involving group theory, ring theory, character theory etc. Having stated details on \cut groups and its generalised classes, we close the article by remarks on two aspects of the study of \cut groups.
\begin{description}
		\item[Existence] As already stated, finite \cut groups are quite many in numbers, especially relative to rational groups. However, if we go to bigger classes like semi-rational or quadratic rational, we do not see much escalation in the percentage of groups.  This has been well analysed in numbers in \cite{dRV24}.
	\item[Generalisations] In Section 4, we have studied some classes of groups which contain rational groups and generalise the notion of \cut groups. The class of \cut groups can be generalised in other directions. For arbitrary (not necessarily finite) \cut groups, analogous results may be found in \cite{MS99}, \cite{DMS05} and \cite{BMP19}. The class of groups $G$ for which $\rho(G)\leq 1$, may be yet another generalisation of this class. Such groups are termed as extended \cut groups and have been studied in \cite{GKM25}. The list of simple groups which are extended \cut groups has been deduced in \cite{BBM20}.  One may consider yet another generalisation by investigating conditions for trivial units in group ring $RG$, where $R$ is not essentially $\mathbb{Z}$.

\end{description}
	\begin{small}
\newcommand{\etalchar}[1]{$^{#1}$}
\providecommand{\bysame}{\leavevmode\hbox to3em{\hrulefill}\thinspace}
\providecommand{\MR}{\relax\ifhmode\unskip\space\fi MR }
\providecommand{\MRhref}[2]{%
	\href{http://www.ams.org/mathscinet-getitem?mr=#1}{#2}
}
\providecommand{\href}[2]{#2}

\end{small}

\end{document}